\documentclass{article}
\usepackage[demo]{graphicx}
\usepackage{float}
\restylefloat*{figure}
\usepackage{comment}
\usepackage{amsmath}
\usepackage{amsthm}
\usepackage{amssymb}
\usepackage{mathrsfs}
\usepackage{textcomp}
\usepackage[colorlinks=true,linkcolor=blue,citecolor=green]{hyperref}
\usepackage[margin=3cm]{geometry}
\usepackage{stmaryrd}
\usepackage{bussproofs}
\usepackage{mathtools}
\usepackage[all]{xy}
\usepackage{todonotes}
\usepackage{forest}
\forestset{smullyan tableaux/.style={for tree={math content},where n children=1{!1.before computing xy={l=\baselineskip},!1.no edge}{},closed/.style={label=below:$\times$},},}
\newcommand{\Ltriangle}{\mathcal{L}_\triangle}

\newcommand{\Gtriangle}{\mathsf{G}\triangle}
\newcommand{\HGtriangle}{\mathcal{H}\Gtriangle}

\newcommand{\Prop}{\mathtt{Prop}}
\newcommand{\coimplies}{\Yleft}
\newcommand{\Gsquare}{\mathsf{G}^2}
\newcommand{\crispKGsquare}{\mathbf{K}\mathsf{G}^{2\mathsf{c}}}
\newcommand{\KGsquare}{\mathbf{K}\mathsf{G}^2}
\newcommand{\biG}{\mathsf{biG}}
\newcommand{\KbiG}{\mathbf{K}\mathsf{biG}}

\newcommand{\KG}{\mathbf{K}\mathsf{G}}
\newcommand{\crispKG}{\mathbf{K}\mathsf{G}^{\mathsf{c}}}
\newcommand{\fuzzyKG}{\mathbf{K}\mathsf{G}^{\mathsf{f}}}
\newcommand{\HKbiG}{\mathcal{H}\KbiG}
\newcommand{\HcrispKbiG}{\mathcal{H}\crispKbiG}
\newcommand{\HcrispKG}{\mathcal{H}\crispKG}
\newcommand{\Sfconst}{\mathsf{Sf}^{\mathbf{0},\mathbf{1}}}

\newcommand{\Th}{\mathsf{Th}}

\newcommand{\HKG}{\mathcal{H}\KG}

\newcommand{\bimodalL}{\mathcal{L}_{\Box,\lozenge}}
\newcommand{\bimodalLtriangle}{\mathcal{L}_{\triangle,\Box,\lozenge}}
\newcommand{\bimodalLtrianglesquare}{\mathcal{L}^\neg_{\triangle,\Box,\lozenge}}
\newcommand{\infobimodalLtrianglesquare}{\mathcal{L}^\neg_{\triangle,\blacksquare,\blacklozenge}}

\newcommand{\fuzzyinfoGsquare}{\mathsf{G}^{2\pm\mathsf{f}}_{\blacksquare,\blacklozenge}}
\newcommand{\infoGsquare}{\mathsf{G}^{2}_{\blacksquare,\blacklozenge}}
\newcommand{\birelinfoGsquare}{\mathsf{G}^{2\pm}_{\blacksquare,\blacklozenge}}
\newcommand{\fuzzybirelKGsquare}{\mathbf{K}\mathsf{G}^{2\pm\mathsf{f}}}
\newcommand{\birelKGsquare}{\mathbf{K}\mathsf{G}^{2\pm}}
\newcommand{\fuzzyKbiG}{{\KbiG}^{\mathsf{f}}}
\newcommand{\HfuzzyKbiG}{\mathcal{H}{\KbiG}^{\mathsf{f}}}
\newcommand{\pspace}{\mathsf{PSpace}}
\newcommand{\monorelinfoGsquare}{\mathsf{G}^{2\mathsf{f}}_{\blacksquare,\blacklozenge}}
\newcommand{\monorelinfoGsquarecrisp}{\mathsf{G}^{2\mathsf{c}}_{\blacksquare,\blacklozenge}}
\newcommand{\fuzzyKGsquare}{\mathbf{K}\mathsf{G}^{2\mathsf{f}}}
\newcommand{\crispKbiG}{{\KbiG}^{\mathsf{c}}}
\newcommand{\crispbirelKGsquare}{\mathbf{K}\mathsf{G}^{2\pm\mathsf{c}}}
\newcommand{\crispinfoGsquare}{\mathsf{G}^{2\pm\mathsf{c}}_{\blacksquare,\blacklozenge}}
\newcommand{\lineBox}{\overline{\Box}}
\newcommand{\linelozenge}{\overline{\lozenge}}
\newcommand{\lineblacksquare}{\overline{\blacksquare}}
\newcommand{\lineblacklozenge}{\overline{\blacklozenge}}
\newcommand{\linebimodalLtrianglesquare}{\overline{\bimodalLtrianglesquare}}
\newcommand{\lineinfobimodalLtrianglesquare}{\overline{\infobimodalLtrianglesquare}}
\newcommand{\Luk}{{\mathchoice{\mbox{\sffamily{\L}}}{\mbox{\sffamily{\L}}}{\mbox{\scriptsize\sffamily{\L}}}{\mbox{\tiny\sffamily{\L}}}}}
\newcommand{\triangletop}{\triangle^\top}
\newtheorem{theorem}{Theorem}[section]
\newtheorem{proposition}{Proposition}[section]
\newtheorem{corollary}{Corollary}[section]
\newtheorem{lemma}{Lemma}[section]
\theoremstyle{definition}
\newtheorem{definition}{Definition}[section]
\theoremstyle{remark}
\newtheorem{remark}{Remark}[section]
\newtheorem{example}{Example}[section]

\newtheorem{convention}{Convention}[section]

\usepackage{longtable}
\usepackage{fancyhdr}
\usepackage{authblk}
\usepackage{soul}

\allowdisplaybreaks
\begin{document}
\providecommand{\keywords}[1]
{
  \small	
  \textbf{\textit{Keywords: }} #1
}
\title{Fuzzy bi-G\"{o}del modal logic and its paraconsistent relatives\thanks{The research of Marta B\'ilkov\'a was supported by the grant 22-01137S of the Czech Science Foundation. The research of Daniil Kozhemiachenko and Sabine Frittella was funded by the grant ANR JCJC 2019, project PRELAP (ANR-19-CE48-0006). This research is part of the MOSAIC project financed by the European Union's Marie Sk\l{}odowska-Curie grant \textnumero101007627. The third author was also supported by the project INTENDED (ANR-19-CHIA-0014).\\\indent The authors wish to thank reviewers for their comments that enhanced the quality of the paper.}}
\author[1]{Marta B\'{\i}lkov\'{a}}
\author[2]{Sabine Frittella}
\author[3]{Daniil Kozhemiachenko}
\affil[1]{The Czech Academy of Sciences, Institute of Computer Science, Czech Republic\\\href{mailto:bilkova@cs.cas.cz}{bilkova@cs.cas.cz}}
\affil[2]{INSA Centre Val de Loire, Univ.\ Orl\'{e}ans, LIFO EA 4022, France\\\href{mailto:sabine.frittella@insa-cvl.fr}{sabine.frittella@insa-cvl.fr}}
\affil[3]{Univ. Bordeaux, CNRS, Bordeaux INP, LaBRI, UMR 5800\\\href{mailto:daniil.kozhemiachenko@u-bordeaux.fr}{daniil.kozhemiachenko@u-bordeaux.fr} (corresponding author)}
\maketitle
\begin{abstract}
We present an axiomatisation of the fuzzy bi-G\"{o}del modal logic $\fuzzyKbiG$ formulated in the language containing $\triangle$ (Baaz Delta operator) and treating $\coimplies$ (coimplication) as the defined connective. We also consider two paraconsistent relatives of $\fuzzyKbiG$ --- $\fuzzybirelKGsquare$ and $\fuzzyinfoGsquare$. These logics are defined on fuzzy frames with two valuations $e_1$ and $e_2$ standing for the support of truth and falsity, respectively, and equipped with \emph{two fuzzy relations} $R^+$ and $R^-$ used to determine supports of truth and falsity of modal formulas. We construct embeddings of $\fuzzybirelKGsquare$ and $\fuzzyinfoGsquare$ into $\fuzzyKbiG$ and use them to obtain the characterisation of $\KGsquare$- and $\infoGsquare$-definable frames. Moreover, we study the transfer of $\fuzzyKbiG$ formulas into $\fuzzybirelKGsquare$, i.e., formulas that are $\fuzzyKbiG$-valid on mono-relational frames $\mathfrak{F}$ and $\mathfrak{F}'$ iff they are $\fuzzybirelKGsquare$-valid on their bi-relational counterparts. Finally, we establish $\pspace$-completeness of all considered logics.

\keywords{G\"{o}del modal logic; paraconsistent modal logics; complexity; frame definability.}
\end{abstract}
\section{Introduction\label{sec:introduction}}
When dealing with modalised propositions such as ‘I believe that Paula has a dog’, ‘I believe that Quinn thinks that Paula has a cat’, ‘Quinn has to submit the report by Friday’, etc., it is reasonable to assume that they can be equipped with a truth degree. After all, we can be more or less convinced of something, have stronger or weaker obligations, etc. On the other hand, it is not always justified that the agents can assign exact numerical values to the degrees of their beliefs or obligations (even if they can compare them). Thus, if we want to formalise these contexts, it makes sense to utilise modal expansions of G\"{o}del logic $\mathsf{G}$. Indeed, $\mathsf{G}$ can be construed as the logic of comparative truth since, in it, the truth of the formula is determined not by the individual values of its variables but rather by the relative order of these values. In particular, the G\"{o}del implication $\phi\rightarrow\chi$ is true iff the value of $\phi$ is less or equal to the value of $\chi$. Note, however, that there is no $\mathsf{G}$-formula of two variables $p$ and $q$ which is true iff the value of $p$ \emph{is strictly greater than} the value of $q$. To express this, one can use the coimplication~$\coimplies$~--- ${\sim\sim}(p\coimplies q)$ (the connective was introduced in~\cite{Rauszer1974}, the notation is due to~\cite{Gore2000}), or the Baaz Delta operator $\triangle$~\cite{Baaz1996}~--- ${\sim}\triangle(p\rightarrow q)$. Adding either of these\footnote{Recall that $\triangle$ and $\coimplies$ are interdefinable modulo $\mathsf{G}$ as follows: $\triangle\phi\coloneqq\mathbf{1}\coimplies(\mathbf{1}\coimplies\phi)$ and $\phi\coimplies\chi\coloneqq\phi\wedge{\sim}\triangle(\phi\rightarrow\chi)$.} to the language of G\"{o}del logic produces the \emph{bi-G\"{o}del logic} $\biG$ (or \emph{symmetric G\"{o}del logic} in the terminology of~\cite{GrigoliaKiseliovaOdisharia2016}).
\paragraph{(bi-)G\"{o}del modal logics}
Our main source of inspiration is the extensive study of G\"{o}del modal logics undertaken in the existing literature. Both $\Box$ and $\lozenge$ fragments\footnote{In contrast to the classical logic, $\Box$ and $\lozenge$ are not interdefinable in G\"{o}del modal logic.} of the G\"{o}del modal logic are axiomatised (cf.~\cite{CaicedoRodriguez2010} for the Hilbert-style axiomatisation and~\cite{MetcalfeOlivetti2009,MetcalfeOlivetti2011} for Gentzen calculi). Furthermore, fuzzy~\cite{CaicedoRodriguez2015} ($\KG^\mathsf{f}$) and crisp~\cite{RodriguezVidal2021} ($\crispKG$) bi-modal logics\footnote{In this text, we use the term ‘fuzzy modal logic’ to denote the logic of \emph{all frames} (i.e., crisp \emph{and} fuzzy); ‘crisp modal logic’ stands for the logic of \emph{only crisp frames}.} are also axiomatised; there is also a~tableaux calculus for $\fuzzyKG$ and $\crispKG$~\cite{Rogger2016phd}. It is also established that all these logics are decidable~\cite{CaicedoMetcalfeRodriguezRogger2013} and, in fact, $\pspace$-complete~\cite{CaicedoMetcalfeRodriguezRogger2017}, even though they are more expressive than $\mathbf{K}$. Regarding the bi-G\"{o}del modal logics, a temporal expansion has been studied in~\cite{AguileraDieguezFernandez-DuqueMcLean2022} and a provability logic with algebraic semantics was presented in~\cite{GrigoliaKiseliovaOdisharia2016}.
\paragraph{Paraconsistent modal logics}
The second source of motivation is the study of \emph{paraconsistent modal logics}. A logic is paraconsistent when its entailment does not satisfy the explosion property~--- $p,\neg p\models q$ (or \emph{ex contradictione quodlibet}). In modal contexts, the failure of this principle aligns very well with our intuition. Indeed, one can have contradictory beliefs but still \emph{not believe in something else}; likewise, even if one has conflicting obligations, it does not mean that they have (or are permitted) to do \emph{everything}.

One of the most common treatments of paraconsistent modal logics is to consider Kripke frames \emph{with two independent valuations} as it is done in, e.g.,~\cite{Wansing2005,Goble2006,Priest2008FromIftoIs,OdintsovWansing2010,OdintsovWansing2017}. These valuations are interpreted as supports of truth and falsity of a formula in a given state. Historically, this idea can be traced to the Belnap--Dunn logic ($\mathsf{BD}$, alias ‘first-degree entailment’ or $\mathsf{FDE}$)~\cite{Dunn1976,Belnap1977computer,Belnap1977fourvalued}. Each proposition could be true, false, both true and false or neither true nor false depending on the available information.

An important consequence of having independent supports of truth and falsity is that it is no longer the case that any two propositions have \emph{comparable values}. Indeed, it might be the case that both truth and falsity of $\phi$ are supported but neither truth nor falsity of $\chi$ is supported. In addition, both the support of the falsity and the support of the truth of $\phi$ can be greater than those of $\chi$ (or vice versa). This makes sense if we wish to model agents who cannot always compare the degrees of their beliefs in two given statements, who cannot always compare the strengths of their obligations, etc. This is justified if, e.g., $\phi$ and $\chi$ do not have any content in common. In addition, in the multi-agent setting, we might not always be able to compare the strength of beliefs of different agents even in the same statement.

After introducing two valuations on a frame, it makes sense to consider frames with \emph{two relations} designated $R^+$ and $R^-$ and utilised to determine supports of truth and falsity of modal formulas (cf., e.g.,~\cite{Sherkhonov2008,Drobyshevich2020} for examples of such logics). Intuitively, if we assume $R^+$ and $R^-$ to be fuzzy (i.e., if $W$ is the set of states of a given frame, then $R^+,R^-:W\times W\rightarrow[0,1]$) and the states in the frame to represent sources that refer to one another, they can be construed as, respectively, degrees of trust of one source in confirmations (assertions) or denials of the other source. For example, if a~source is too sceptical, we might be inclined to trust its denials \emph{less than} we trust its confirmations. On the contrary, if a~sensationalist source denies something, it can be considered \emph{more} trustworthy than when the very same source asserts something.

The idea of associating two independent relations to one modality was also utilised in the knowledge representation. E.g., in~\cite{MaierMaHitzler2013}, paraconsistent semantics of $\mathcal{ALC}$ and its extensions are presented where each role name is associated to two relations on the carrier. This way, the authors managed to maintain paraconsistent behaviour of the knowledge bases when using role constructors (in particular, reflexive and irreflexive closures), self restrictions, and negative role assertions.
\paragraph{Logics in the paper}
This paper continues the study of paraconsistent expansions of G\"{o}del modal logics that we began in~\cite{BilkovaFrittellaKozhemiachenko2022IJCAR} and continued in~\cite{BilkovaFrittellaKozhemiachenko2023IGPL,BilkovaFrittellaKozhemiachenko2023nonstandard}. We cover quite a few logics (cf.~Fig.~\ref{fig:logics}). Previously, we were mostly dealing with bi-G\"{o}del and paraconsistent G\"{o}del logics on \emph{crisp} frames. Now, our main interest lies in the logics on fuzzy frames (put in rectangles in the picture). These logics are obtained from the fuzzy bi-modal G\"{o}del logic ($\fuzzyKbiG$ from~\cite{BilkovaFrittellaKozhemiachenko2022IJCAR}) by adding a~De~Mor\-gan negation $\neg$ or (additionally) replacing $\Box$ and $\lozenge$ with $\blacksquare$ and $\blacklozenge$.
\begin{figure}[ht]
\centering
\begin{minipage}[b]{.35\textwidth}
\centering
\begin{tikzpicture}[>=stealth,relative]
\node (U1) at (0,-1.5) {$(0,1)$};
\node (U2) at (-1.5,0) {$(0,0)$};
\node (U3) at (1.5,0) {$(1,1)$};
\node (U4) at (0,1.5) {$(1,0)$};
\node (tmin) at (-2.5,-1.5) {};
\node (tmax) at (-2.5,1.5) {};
\node (imin) at (-1.5,-2.5) {};
\node (imax) at (1.5,-2.5) {};
\path[-,draw] (U1) to (U2);
\path[-,draw] (U1) to (U3);
\path[-,draw] (U2) to (U4);
\path[-,draw] (U3) to (U4);
\draw[dashed] (U1) -- (U4);
\draw[dashed] (U2) -- (U3);
\draw[->,draw] (tmin)  edge node[above,rotate=90] {truth}  (tmax);
\draw[->,draw] (imin)  edge node[below] {information} (imax);
\end{tikzpicture}
\caption{$[0,1]^{\Join}$: the truth order goes upwards and the information order goes rightwards.}
\label{fig:01squareagain}
\end{minipage}
\hfill
\begin{minipage}[b]{.6\textwidth}
\centering
\normalsize{\[\xymatrix{
&*+[F]{\fuzzybirelKGsquare}&\\
&\fuzzyKGsquare\ar[u]^{\pm}&{\crispbirelKGsquare}\ar[ul]^{\mathsf{ff}}&*+[F]{\fuzzyinfoGsquare}\\
*+[F]{\fuzzyKbiG}\ar[ur]|{\neg}&&\crispKGsquare\ar[ul]^{\mathsf{ff}}\ar[u]^{\pm}&\monorelinfoGsquare\ar[u]^{\pm}&{\crispinfoGsquare}\ar[ul]^{\mathsf{ff}}\\
\crispKbiG\ar[urr]|{\neg}\ar[u]^{\mathsf{ff}}&\fuzzyKG\ar[ul]|{\coimplies/\triangle}&&\monorelinfoGsquarecrisp\ar[u]^{\mathsf{ff}}\ar[ur]^{\pm}\\
\mathsf{biG}\ar[u]|{\Box,\lozenge}&\crispKG\ar[uur]|{\neg}\ar[u]^{\mathsf{ff}}\ar[ul]|{\coimplies/\triangle}&\Gsquare\ar[uu]|{\Box/\lozenge}\ar[ur]|{\blacksquare,\blacklozenge}\\
&\mathsf{G}\ar[ul]|{\coimplies/\triangle}\ar[u]|{\Box,\lozenge}\ar[ur]|{\neg}
}\]}
\caption{Logics mentioned in the paper. $\mathsf{ff}$~stands for ‘permitting fuzzy frames’; $\pm$ for ‘permitting birelational frames’. Subscripts on arrows denote language expansions. $/$ stands for ‘or’ and comma for ‘and’. The logics we mainly focus on in this paper are \fbox{put in frames}.}
\label{fig:logics}
\end{minipage}
\end{figure}

Let us quickly explain the difference between these two pairs of modalities. Recall, first of all (cf.~Fig.~\ref{fig:01squareagain}), that instead of giving two independent valuations on $[0,1]$, we can equivalently define one valuation on $[0,1]^{\Join}=[0,1]\times[0,1]^{\mathsf{op}}$. $\fuzzybirelKGsquare$ uses $\Box$ and $\lozenge$ which can be thought of as generalisations of the meet and join on $[0,1]^{\Join}$ w.r.t.\ the truth order: i.e., the value of $\Box\phi$ is computed using $\bigwedge$ --- the truth infimum of $\phi$ in the accessible states while $\lozenge$ uses $\bigvee$ --- the truth supremum. $[0,1]^{\Join}$, however, is a \emph{bi-lattice} and thus, there are $\sqcap$ and $\sqcup$, i.e., meet and join w.r.t.\ the informational order. Thus, $\fuzzyinfoGsquare$ (first introduced in~\cite{BilkovaFrittellaKozhemiachenko2023nonstandard}) expands $\biG$ with $\neg$ and $\blacksquare$ and $\blacklozenge$ that generalise $\sqcap$ and $\sqcup$ (we will further call $\blacksquare$ and $\blacklozenge$ ‘informational modalities’ and $\Box$ and $\lozenge$ ‘standard modalities’) as expected: the value of $\blacksquare\phi$ is computed using $\bigsqcap$ (the informational infimum) of the values of $\phi$ in the accessible states and the value of $\blacklozenge$ is obtained via $\bigsqcup$ (the informational supremum).
\begin{convention}\label{conv:logicsnotation}
In what follows, we are going to use several conventions for naming the logics.
\begin{itemize}
\item Index $^2$ designates that the logic is evaluated on frames with two valuations.
\item Indices $^\mathsf{c}$ and $^\mathsf{f}$ stand for, respectively, logics on crisp and fuzzy frames.
\item Index $^\pm$ denotes the logic whose frames have two accessibility relations $R^+$ and $R^-$.
\end{itemize}
When we consider both fuzzy and crisp versions of a given logic at once, we are going to omit the corresponding indices. Likewise, we omit all indices except $^2$ when we deal with all versions of a logic. Thus, for example ‘in all $\KGsquare$'s it holds that $X$’ means that $X$ is true w.r.t.\ $\crispKGsquare$, $\fuzzyKGsquare$, $\crispbirelKGsquare$, and $\fuzzybirelKGsquare$. Similarly, ‘$Y$ holds in $\birelinfoGsquare$’ stands for ‘$Y$ holds in $\fuzzyinfoGsquare$ and $\crispinfoGsquare$’.
\end{convention}

\paragraph{Plan of the paper}
In the previous paper~\cite{BilkovaFrittellaKozhemiachenko2023IGPL}, we studied the \emph{crisp} modal bi-G\"{o}del logic\footnote{I.e., the logic defined on the frames where the accessibility relation is crisp.} and its paraconsistent expansion (in this paper, we denote them $\crispKbiG$ and $\crispKGsquare$, respectively). We also studied its computational and model-theoretic properties. The main goal of this paper is to study the fuzzy bi-G\"{o}del modal logic and its paraconsistent relatives with standard and informational modalities. The remainder of the text is organised as follows.

In Section~\ref{sec:preliminaries}, we remind the semantics of propositional bi-G\"{o}del logic and its expansion with $\Box$ and~$\lozenge$ interpreted on both crisp and fuzzy frames. In addition, we show that some classes of frames definable in $\KbiG$ cannot be defined without $\triangle$. Section~\ref{sec:HKbiGf} is dedicated to the axiomatisation of $\fuzzyKbiG$. We construct a~Hilbert-style calculus for $\fuzzyKbiG$ and establish its strong completeness by adapting the proof from~\cite{CaicedoRodriguez2015}. In Section~\ref{sec:paraconsistentlogics}, we present $\KGsquare$'s and $\infoGsquare$'s, establish their embeddings into $\KbiG$, and study their semantical properties. Section~\ref{sec:transfer} is dedicated to the study of transferrable formulas, i.e., those that are valid on the same frames both in $\KbiG$ and $\KGsquare$. In Section~\ref{sec:complexity}, we establish the $\pspace$-completeness of $\fuzzyKbiG$ using the technique from~\cite{CaicedoMetcalfeRodriguezRogger2017}. We then use the embeddings obtained in Section~\ref{sec:paraconsistentlogics} to obtain the $\pspace$-completeness of the paraconsistent relatives of $\fuzzyKbiG$. Finally, Section~\ref{sec:conclusion} is devoted to the discussion of the results that we obtained and outlines further work.
\section{Preliminaries\label{sec:preliminaries}}
To make the paper self-contained, we begin with the presentation of the propositional fragment of $\KbiG$, namely, the bi-G\"{o}del logic $\biG$.
\subsection{Propositional bi-G\"{o}del logic}
The language $\Ltriangle$ is generated from the countable set $\Prop$ via the following grammar.
\begin{align*}
\Ltriangle\ni\phi&\coloneqq p\in\Prop\mid{\sim}\phi\mid\triangle\phi\mid(\phi\wedge\phi)\mid(\phi\vee\phi)\mid(\phi\rightarrow\phi)
\end{align*}
We also introduce two defined constants
\begin{align*}
\mathbf{1}&\coloneqq p\rightarrow p&
\mathbf{0}&\coloneqq{\sim}\mathbf{1}
\end{align*}
We choose $\triangle$ over $\coimplies$ as a primitive symbol because the former allows for a shorter and more elegant axiomatisation of the propositional fragment. Furthermore, the use of $\triangle$ simplifies the completeness proof of $\fuzzyKbiG$.

The semantics of bi-G\"{o}del logic $\biG$ is given in the following definition. For the sake of simplicity, we also include $\coimplies$ in the definition of the bi-G\"{o}del algebra on $[0,1]$ since it will simplify the presentation of the semantics of paraconsistent logics.
\begin{definition}\label{def:biGalgebra}
The bi-G\"{o}del algebra on $[0,1]$ denoted $[0,1]_{\biG}=\langle[0,1],0,1,\wedge_\mathsf{G},\vee_\mathsf{G},\rightarrow_{\mathsf{G}},\coimplies,\sim_\mathsf{G},\triangle_\mathsf{G}\rangle$ is defined as follows: for all $a,b\in[0,1]$, the standard operations are given by $a\wedge_\mathsf{G}b\coloneqq\min(a,b)$, $a\vee_\mathsf{G}b\coloneqq\max(a,b)$,
\begin{align*}
a\rightarrow_\mathsf{G}b&=
\begin{cases}
1\text{ if }a\leq b\\
b\text{ else}
\end{cases}
&
a\coimplies_\mathsf{G}b&=
\begin{cases}
0\text{ if }a\leq b\\
a\text{ else}
\end{cases}
&
{\sim}_\mathsf{G}a&=
\begin{cases}
0\text{ if }a>0\\
1\text{ else}
\end{cases}
&
\triangle_\mathsf{G}a&=
\begin{cases}
0\text{ if }a<1\\
1\text{ else}
\end{cases}
\end{align*}
A \emph{$\biG$ valuation} is a homomorphism $e:\Ltriangle\rightarrow[0,1]_\biG$ that is defined for the complex formulas as $e(\phi\circ\phi')=e(\phi)\circ_\mathsf{G}e(\phi')$ for every binary connective $\circ$ and $e(\sharp\phi)=\sharp_\mathsf{G}(e(\phi))$ for $\sharp\in\{\sim,\triangle\}$. We say that $\phi$ is \emph{valid} iff $e(\phi)=1$ under every valuation. Moreover, $\Gamma\subseteq\Ltriangle$ \emph{entails} $\chi\in\Ltriangle$ ($\Gamma\models_{\biG}\chi$) iff for every valuation $e$, it holds that
\[\inf\{e(\phi):\phi\in\Gamma\}\leq e(\chi).\]
\end{definition}
It is now easy to see that $e(\triangle\phi)=e(\mathbf{1}\coimplies(\mathbf{1}\coimplies\phi))$ and $e(\phi\coimplies\chi)=e(\phi\wedge{\sim}\triangle(\phi\rightarrow\chi))$ for every $e$ as intended.

Finally, let us recall the Hilbert-style calculus for $\biG$ from~\cite{Baaz1996}.
\begin{definition}[$\HGtriangle$ --- Hilbert-style calculus for $\biG$]
The calculus has the following axiom schemas and rules (for any $\phi$, $\chi$,~$\psi$):
\begin{enumerate}
\item $(\phi\rightarrow\chi)\rightarrow((\chi\rightarrow\psi)\rightarrow(\phi\rightarrow\psi))$
\item $\phi\rightarrow(\phi\vee\chi)$; $\chi\rightarrow(\phi\vee\chi)$; $(\phi\rightarrow\psi)\rightarrow((\chi\rightarrow\psi)\rightarrow((\phi\vee\chi)\rightarrow\psi))$
\item $(\phi\wedge\chi)\rightarrow\phi$; $(\phi\wedge\chi)\rightarrow\chi$; $(\phi\rightarrow\chi)\rightarrow((\phi\rightarrow\psi)\rightarrow(\phi\rightarrow(\chi\wedge\psi)))$
\item $(\phi\rightarrow(\chi\rightarrow\psi))\rightarrow((\phi\wedge\chi)\rightarrow\psi)$; $((\phi\wedge\chi)\rightarrow\psi)\rightarrow(\phi\rightarrow(\chi\rightarrow\psi))$
\item ${\sim}\phi\rightarrow(\phi\rightarrow\chi)$
\item $(\phi\rightarrow\chi)\vee(\chi\rightarrow\phi)$
\item $\triangle\phi\vee{\sim}\triangle\phi$
\item $\triangle(\phi\rightarrow\chi)\rightarrow(\triangle\phi\rightarrow\triangle\chi)$; $\triangle(\phi\vee\chi)\rightarrow(\triangle\phi\vee\triangle\chi)$
\item $\triangle\phi\rightarrow\phi$; $\triangle\phi\rightarrow\triangle\triangle\phi$
\item[MP] $\dfrac{\phi\quad\phi\rightarrow\chi}{\chi}$
\item[$\triangle$nec] $\dfrac{\vdash\phi}{\vdash\triangle\phi}$
\end{enumerate}
Let $\Gamma\subseteq\Ltriangle$. We define the notion \emph{$\psi$ depends on $\Gamma$ in the derivation} by the following recursion:
\begin{itemize}
\item if $\psi\in\Gamma$, then $\psi$ depends on $\Gamma$;
\item if $\chi$ depends on $\Gamma$ and is a~premise of the rule by which $\psi$ is obtained, then $\psi$ depends on $\Gamma$.
\end{itemize}

We can now define the derivability relation $\vdash_{\HGtriangle}$. $\Gamma\subseteq\Ltriangle$ \emph{derives} $\chi$ ($\Gamma\vdash_{\HGtriangle}\chi$) iff there is a finite sequence of formulas $\phi_1,\ldots,\phi_n,\chi$ each of which is a member of $\Gamma$, an instance of an axiom schema, obtained from the previous ones by MP, or obtained by an application of $\triangle$nec to a~previous formula that does not depend on $\Gamma$. If $\varnothing\vdash_{\HGtriangle}\phi$, we write $\HGtriangle\vdash\phi$ and say that $\phi$ is \emph{provable} or that $\phi$ is \emph{a~theorem of $\HGtriangle$}.
\end{definition}
\begin{remark}\label{rem:trianglededuction}
Note that it is crucial for the soundness of $\HGtriangle$ that $\triangle$nec is applied only to theorems. Otherwise, $p\vdash_{\HGtriangle}\triangle p$ would be derivable which, of course, is not a~valid instance of entailment.
\end{remark}
\begin{proposition}[{\cite[Theorem~3.1]{Baaz1996}}]\label{prop:HbiGcompleteness}
$\HGtriangle$ is strongly complete: for any $\Gamma\cup\{\phi\}\subseteq\Ltriangle$, it holds that
\[\Gamma\vdash_{\HGtriangle}\phi\text{ iff }\Gamma\models_{\biG}\phi\]
\end{proposition}
\begin{convention}
In what follows, given a calculus $\mathcal{H}$, we use $\Th(\mathcal{H})$ to denote the set of its \emph{theorems}, i.e., formulas that can be proven in $\mathcal{H}$ without assumptions.
\end{convention}
\subsection{$\KbiG$: language and semantics}
Let us now introduce the logic $\KbiG$. The language $\bimodalLtriangle$ of $\KbiG$ expands $\Ltriangle$ with two additional modal operators: $\Box$ and $\lozenge$. We will further use $\bimodalL$ to denote the $\triangle$-free fragment of $\bimodalLtriangle$ and $\KG$ to denote the G\"{o}del modal logic, i.e., the $\bimodalL$ fragment of $\KbiG$.
\begin{definition}[Frames]\label{def:frames}~
\begin{itemize}
\item A \emph{fuzzy frame} is a tuple $\mathfrak{F}=\langle W,R\rangle$ with $W\neq\varnothing$ and $R:W\times W\rightarrow[0,1]$.
\item A \emph{crisp frame} is a tuple $\mathfrak{F}=\langle W,R\rangle$ with $W\neq\varnothing$ and $R\subseteq W\times W$ (or, equivalently, with $R:W\times W\rightarrow\{0,1\}$).
\end{itemize}
\end{definition}
\begin{definition}[$\KbiG$ semantics]\label{def:KbiGsemantics}
A \emph{$\KbiG$ model} is a tuple $\mathfrak{M}=\langle W,R,e\rangle$ with $\langle W,R\rangle$ being a~(crisp or fuzzy) frame, and $e:\Prop\times W\rightarrow[0,1]$. $e$ (a valuation) is extended on complex $\bimodalLtriangle$ formulas according to Definition~\ref{def:biGalgebra} in the cases of propositional connectives:
\begin{align*}
e(\sharp\phi,w)&=\sharp_\mathsf{G}(e(\phi,w))&e(\phi\circ\phi',w)&=e(\phi,w)\circ_\mathsf{G}e(\phi',w)\tag{$\sharp\in\{\sim,\triangle\}$, $\circ\in\{\wedge,\vee,\rightarrow\}$}
\end{align*}
The interpretation of modal formulas is as follows:
\begin{align*}
e(\Box\phi,w)&=\inf\limits_{w'\in W}\{wRw'\rightarrow_\mathsf{G}e(\phi,w')\}&e(\lozenge\phi,w)&=\sup\limits_{w'\in W}\{wRw'\wedge_\mathsf{G}e(\phi,w')\}.
\end{align*}
We say that $\phi\in\bimodalLtriangle$ is \emph{$\KbiG$-valid on a~pointed frame $\langle\mathfrak{F},w\rangle$} ($\mathfrak{F},w\models_{\KbiG}\phi$) iff $e(\phi,w)=1$ for any model $\mathfrak{M}$ on $\mathfrak{F}$. We call $\phi$ \emph{$\KbiG$-valid on frame $\mathfrak{F}$} ($\mathfrak{F}\models_{\KbiG}\phi$) iff $\mathfrak{F},w\models_{\KbiG}\phi$ for any $w\in\mathfrak{F}$.

$\Gamma$ \emph{entails} $\chi$ (on $\mathfrak{F}$), denoted $\Gamma\models_{\KbiG}\chi$ ($\Gamma\models^{\mathfrak{F}}_{\KbiG}\chi$), iff for every model $\mathfrak{M}$ (on $\mathfrak{F}$) and every $w\in\mathfrak{M}$, it holds that
\[\inf\{e(\phi,w):\phi\in\Gamma\}\leq e(\chi,w).\]
\end{definition}
\begin{convention}
Let $\mathfrak{F}=\langle W,R\rangle$ be a crisp frame and $\mathfrak{H}=\langle U,S\rangle$ a fuzzy frame. We set
\begin{align*}
R(w)&=\{w':wRw'\}&S(u)&=\{u':uSu'>0\}
\end{align*}
\end{convention}

We have already discussed in the introduction that $\triangle$ can be used to define \emph{strict} order on values (while $\rightarrow$ can only define \emph{non-strict} order). In addition, $\triangle$ enhances the expressivity of $\Box$ and $\lozenge$ fragments of $\KG$. In particular, while the $\lozenge$ (fuzzy) fragment of $\KG$ has the finite model property~\cite[Theorem~7.1]{CaicedoRodriguez2010}, $\triangle\lozenge p\rightarrow\lozenge\triangle p$ \emph{does not have} finite countermodels~\cite[Proposition~2.10.2]{BilkovaFrittellaKozhemiachenko2023IGPL}. Likewise, while the $\Box$ fragment of $\KG$ is complete both w.r.t.\ crisp and fuzzy frames~\cite[Theorem~4.2]{CaicedoRodriguez2010}, $\triangle\Box p\rightarrow\Box\triangle p$ \emph{defines} crisp frames~\cite[Proposition~2.10.1]{BilkovaFrittellaKozhemiachenko2023IGPL}. Moreover, we can show that there are frame properties that cannot be defined without $\triangle$ in the bi-modal $\KG$.
\begin{proposition}\label{prop:triangleonlyfuzzymodels}
There are $\KbiG$-definable classes of fuzzy frames that are not $\KG$-definable
\end{proposition}
\begin{proof}
Consider $\tau={\sim}\triangle\lozenge\mathbf{1}\wedge{\sim}\Box\mathbf{0}$. It is clear that
\begin{align}
\mathfrak{F}\models_{\KbiG}\tau&\text{ iff }\forall u\in\mathfrak{F}:0<\sup\{uRu':u'\in\mathfrak{F}\}<1
\label{equ:properfuzzyseriality}
\end{align}
Denote the class of frames satisfying~\eqref{equ:properfuzzyseriality} with $\mathbb{S}$. Observe now that $\mathfrak{F}\in\mathbb{S}$ iff $e(\lozenge\mathbf{1}\vee\Box\mathbf{0},w)<1$ in every $w\in\mathfrak{F}$ since it is always the case that $e(\Box\mathbf{0},w),e({\sim}\triangle\lozenge\mathbf{1},w)\in\{0,1\}$ and since $e({\sim}\triangle\lozenge\mathbf{1},w)=1$ iff $e(\lozenge\mathbf{1},w)<1$. On the other hand, $\lozenge\mathbf{1}$ can take any value from $[0,1]$ on a~fuzzy frame. But there is no $\bimodalL$ formula that is true iff $e(\lozenge\mathbf{1}\vee\Box\mathbf{0},w)<1$ because $\lozenge\mathbf{1}\vee\Box\mathbf{0}$ can have any value from $0$ to $1$; thus to express that it has a value less than $1$, one needs $\triangle$ or $\coimplies$. Thus, $\mathbb{S}$ is not $\KG$-definable.
\end{proof}
\section{Axiomatisation of the fuzzy $\KbiG$\label{sec:HKbiGf}}
Before proceeding to the axiomatisation of $\fuzzyKbiG$, let us recall the axiomatisation of the crisp $\KbiG$ from~\cite{BilkovaFrittellaKozhemiachenko2023IGPL}.
\begin{definition}[$\HcrispKbiG$ --- Hilbert-style calculus for $\crispKbiG$]\label{def:HcrispKbiG}
The calculus has the following axiom schemas and rules.
\begin{description}
\item[$\biG$:] All substitution instances of $\HGtriangle$ theorems and rules.
\item[$\mathbf{0}$:] ${\sim}\lozenge\mathbf{0}$
\item[K:] $\Box(\phi\rightarrow\chi)\rightarrow(\Box\phi\rightarrow\Box\chi)$; $\lozenge(\phi\vee\chi)\rightarrow(\lozenge\phi\vee\lozenge\chi)$
\item[FS:] $\lozenge(\phi\rightarrow\chi)\rightarrow(\Box\phi\rightarrow\lozenge\chi)$; $(\lozenge\phi\rightarrow\Box\chi)\rightarrow\Box(\phi\rightarrow\chi)$
\item[${\sim}\triangle\lozenge$:] ${\sim}\triangle(\lozenge\phi\rightarrow\lozenge\chi)\rightarrow\lozenge{\sim}\triangle(\phi\rightarrow\chi)$
\item[Cr:] $\Box(\phi\vee\chi)\rightarrow(\Box\phi\vee\lozenge\chi)$; $\triangle\Box\phi\rightarrow\Box\triangle\phi$
\item[nec:] $\dfrac{\vdash\phi}{\vdash\Box\phi}$; $\dfrac{\vdash\phi\rightarrow\chi}{\vdash\lozenge\phi\rightarrow\lozenge\chi}$
\end{description}
\end{definition}

In the list above, $\mathbf{FS}$ denotes two Fisher Servi axioms~\cite{FisherServi1984}. In (super-)intuitionistic modal logics, they guarantee that $\Box$ and $\lozenge$ are defined w.r.t.\ the same accessibility relation.

It is also clear that \textbf{Cr} axioms are not $\fuzzyKbiG$-valid since they define crisp frames. It is also easy to check that ${\sim}\triangle\lozenge$ is valid only on crisp frames as well. Indeed, consider Fig.~\ref{fig:simtrianglediamondcounterexample}: $e(\lozenge{\sim}\triangle(p\rightarrow q),w_0)=\frac{1}{2}$ (since $w_0Rw_1=\frac{1}{2}$) but $e({\sim}\triangle(\lozenge p\rightarrow\lozenge q),w_0)=1$.
\begin{figure}
\centering
\[\xymatrix{w_0\ar[r]^(.3){\frac{1}{2}}&w_1:\txt{$p=\frac{1}{3}$\\$q=\frac{1}{4}$}}\]
\caption{A fuzzy model falsifying ${\sim}\triangle(\lozenge p\rightarrow\lozenge q)\rightarrow\lozenge{\sim}\triangle(p\rightarrow q)$ at $w_0$.}
\label{fig:simtrianglediamondcounterexample}
\end{figure}
In fact, the following statement holds.
\begin{proposition}\label{prop:simtrianglediamondredundancy}
${\sim}\triangle\lozenge$ is redundant in $\HcrispKbiG$.
\end{proposition}
\begin{proof}
First, observe that $\lozenge{\sim}p\leftrightarrow{\sim\sim}\lozenge{\sim}p$ is $\crispKG$-valid (and thus, $\HcrispKG$-provable). Thus, it suffices to check that ${\sim}\triangle(\lozenge p\rightarrow\lozenge q)\rightarrow{\sim\sim}\lozenge{\sim}\triangle(p\rightarrow q)$ is provable. Now recall that ${\sim}\phi\rightarrow{\sim}\chi$ is $\HGtriangle$-provable iff ${\sim\sim}\chi\rightarrow{\sim\sim}\phi$ is provable. Thus, we need to prove ${\sim\sim\sim}\lozenge{\sim}\triangle(p\rightarrow q)\rightarrow{\sim\sim}\triangle(\lozenge p\rightarrow\lozenge q)$. But $\HGtriangle\vdash{\sim\sim\sim}\phi\leftrightarrow{\sim}\phi$ and $\HGtriangle\vdash{\sim\sim}\triangle\phi\leftrightarrow\triangle\phi$, whence, we reduce our task to the proof of ${\sim}\lozenge{\sim}\triangle(p\rightarrow q)\rightarrow\triangle(\lozenge p\rightarrow\lozenge q)$. Finally, $\HKG\vdash{\sim}\lozenge\phi\leftrightarrow\Box{\sim}\phi$, whence, we need to prove $\Box\triangle(p\rightarrow q)\rightarrow\triangle(\lozenge p\rightarrow\lozenge q)$.

To do this, recall~\cite[Proposition~3.1]{BilkovaFrittellaKozhemiachenko2023IGPL} that $\Box\triangle\phi\rightarrow\triangle\Box\phi$ is provable in $\HcrispKbiG$ \emph{without the use of $\mathbf{Cr}$ and ${\sim}\triangle\lozenge$} and that $\HKG\vdash\Box(p\rightarrow q)\rightarrow(\lozenge p\rightarrow\lozenge p)$ (whence, $\triangle\Box(p\rightarrow q)\rightarrow\triangle(\lozenge p\rightarrow\lozenge p)$ is provable using $\triangle$nec). Thus, using the transitivity of $\rightarrow$, we obtain $\Box\triangle(p\rightarrow q)\rightarrow\triangle(\lozenge p\rightarrow\lozenge q)$, as required.
\end{proof}

The above statement makes one inquire whether the axiomatisation of $\fuzzyKbiG$ is obtained by removing $\mathbf{Cr}$ from $\HcrispKbiG$. In the remainder of the section, we show that this is indeed the case.
\begin{definition}[$\HfuzzyKbiG$ --- Hilbert-style calculus for $\fuzzyKbiG$]\label{def:HKfuzzybiG}
The calculus has the following axiom schemas and rules.
\begin{description}
\item[$\biG$:] All substitution instances of $\HGtriangle$ theorems and rules.
\item[$\mathbf{0}$:] ${\sim}\lozenge\mathbf{0}$
\item[K:] $\Box(\phi\rightarrow\chi)\rightarrow(\Box\phi\rightarrow\Box\chi)$; $\lozenge(\phi\vee\chi)\rightarrow(\lozenge\phi\vee\lozenge\chi)$
\item[FS:] $\lozenge(\phi\rightarrow\chi)\rightarrow(\Box\phi\rightarrow\lozenge\chi)$; $(\lozenge\phi\rightarrow\Box\chi)\rightarrow\Box(\phi\rightarrow\chi)$
\item[nec:] $\dfrac{\vdash\phi}{\vdash\Box\phi}$; $\dfrac{\vdash\phi\rightarrow\chi}{\vdash\lozenge\phi\rightarrow\lozenge\chi}$
\end{description}
\end{definition}
Note that we do not add any modal axioms to the calculus for $\fuzzyKG$ (the only axioms we add to $\HKG$ are the propositional $\triangle$ axioms of $\HGtriangle$). Moreover (just as it is the case with $\HKG$), since modal rules can only be applied to theorems, it is clear that we can reduce $\HfuzzyKbiG$ proofs to propositional $\HGtriangle$ proofs using $\HfuzzyKbiG$ theorems as additional assumptions:
\begin{align}
\label{equ:propositionalreduction}
\Gamma\vdash_{\HfuzzyKbiG}\phi&\text{ iff }\Gamma,\Th(\HfuzzyKbiG)\vdash_{\HGtriangle}\phi
\end{align}
It is also easy to see that the deduction theorem holds for $\HfuzzyKbiG$:
\begin{align}
\label{equ:HfuzzyKbiGdeductiontheorem}
\Gamma,\phi\vdash_{\HfuzzyKbiG}\chi&\text{ iff }\Gamma\vdash_{\HfuzzyKbiG}\phi\rightarrow\chi
\end{align}

This is why our completeness proof utilises the same method as the proof of $\HKG$ completeness in~\cite{CaicedoRodriguez2015}. There is, however, an important difference between $\fuzzyKG$ and $\fuzzyKbiG$. Namely, the $\KG$ entailment is defined via the preservation of $1$ from premises to the assumption in all valuations~\cite[Definition~1.1]{CaicedoRodriguez2015}. However, the entailment as preservation of $1$ is not equivalent to the entailment as preservation of the order on $[0,1]$ in $\KbiG$. Indeed, 
$p\not\models_{\KbiG}\triangle p$, even though $e(\triangle p,w)=1$ in every valuation $e$ s.t.\ $e(p,w)=1$.

We proceed as follows. First, we prove the weak completeness of $\HfuzzyKbiG$ via a canonical model construction and the truth lemma. Then, we provide a translation into the classical first-order logic and use compactness to achieve the strong completeness result. The canonical model definition is the same as in~\cite{CaicedoRodriguez2015}.
\begin{definition}[Canonical (counter-)model of a formula]\label{def:canonicalmodel}
Let $\tau\in\bimodalLtriangle$ be s.t.\ $\HfuzzyKbiG\nvdash\tau$ and let $\Sfconst(\tau)=\{\mathbf{0},\mathbf{1}\}\cup\{\psi:\psi\text{ occurs in }\tau\text{ as a subformula}\}$. We define \emph{the canonical countermodel} of $\tau$, $\mathfrak{M}^\tau=\langle W^\tau,\mathsf{R}^\tau,e^\tau\rangle$ as follows.
\begin{itemize}
\item $W^\tau$ is the set of all $\Gtriangle$ homomorphisms $u:\bimodalLtriangle\rightarrow[0,1]_{\biG}$ s.t.\ all theorems of $\HfuzzyKbiG$ are evaluated at $1$.
\item $u\mathsf{R}^\tau u'=\inf\left\{(u(\Box\psi)\rightarrow_\mathsf{G}u'(\psi))\wedge_\mathsf{G}(u'(\psi)\rightarrow_\mathsf{G}u(\lozenge\psi)):\psi\in\Sfconst(\tau)\right\}$.
\item $e^\tau(p,u)=u(p)$.\end{itemize}
\end{definition}

In the definition above, we interpret ‘$\Gtriangle$-homomorphisms’ as the valuations of \emph{propositional} formulas built from variables and \emph{modalised formulas of the form $\Box\chi$ or $\lozenge\chi$}. We need to explicitly demand that all theorems be evaluated at $1$ since, for example, $\Box(p\rightarrow p)$ counts as a~‘variable’ in this setting. But it is a theorem of $\HfuzzyKbiG$, and thus, has to be evaluated at $1$ in every state of $W^\tau$. The definition of $\mathsf{R}^\tau$ is a~modification of the definition of the accessibility relation in the canonical models of classical modal logics (cf., e.g.,~\cite[Definition~4.18]{BlackburndeRijkeVenema2010}) for the fuzzy case.

The most important part of the truth lemma is the cases of modal formulas. The proofs follow the original ones of~\cite[Claim~1]{CaicedoRodriguez2015} (for the $\Box$ formulas) and~\cite[Claim~2]{CaicedoRodriguez2015} (for the $\lozenge$~for\-mu\-las) but to make the text self-contained, we give them here in full and then discuss the main differences between them and the original ones.
\begin{lemma}\label{lemma:truthlemmaBox}
Let $\alpha<1$, $\varepsilon>0$, and $v(\Box\phi)=\alpha$. Then there is $w\in W^\tau$ s.t.\ 
$w(\phi)<\alpha+\varepsilon$ and $v\mathsf{R}^\tau w>w(\phi)$.
\end{lemma}
\begin{proof}
The proof follows that of~\cite[Claim~1]{CaicedoRodriguez2015}. First, we produce $u\in W$ s.t.\ $u(\phi)<1$ and code the ordering conditions for $w$ with a~theory $\Gamma_{\phi,v}$. Then, we move the values $u(\theta)$ ($\theta\in\Sfconst$) to the correct valuation $w$ by composing $u$ with an increasing map of $[0,1]$ into itself.

Now, define
\begin{align}
\label{equ:truthlemmaBoxrefutingtheory}
\Gamma_1&=\{\theta:v(\Box\theta)>\alpha\text{ and }\theta\in\Sfconst(\tau)\}\nonumber\\
\Gamma_2&=\{\theta_1\rightarrow\theta_2:v(\lozenge\theta_1)\leq v(\Box\theta_2)\text{ and }\theta_1,\theta_2\in\Sfconst(\tau)\}\nonumber\\
\Gamma_3&=\{(\theta_2\rightarrow\theta_1)\rightarrow\theta_1:v(\lozenge\theta_1)<v(\Box\theta_2)\text{ and }\theta_1,\theta_2\in\Sfconst(\tau)\}\nonumber\\
\Gamma_{\phi,v}&=\Gamma_1\cup\Gamma_2\cup\Gamma_3
\end{align}
It is clear that $v(\Box\gamma)>\alpha$ for every $\gamma\in\Gamma_{\phi,v}$. Indeed, this holds w.r.t.\ $\Gamma_1$ by construction; for $\Gamma_2$, since $(\lozenge\theta_1\rightarrow\Box\theta_2)\rightarrow\Box(\theta_1\rightarrow\theta_2)$ is an instance of an axiom scheme and since $v((\lozenge\theta_1\rightarrow\Box\theta_2))=1$; for $\Gamma_3$, since $(\Box\theta_2\rightarrow\lozenge\theta_1)\vee\Box((\theta_2\rightarrow\theta_1)\rightarrow\theta_1)$ is $\fuzzyKG$-valid (whence, $\HKG$-provable and evaluated at $1$ in every state of the canonical model) but $v(\Box\theta_2\rightarrow\lozenge\theta_1)<1$, and thus, $v(\Box((\theta_2\rightarrow\theta_1)\rightarrow\theta_1))=1$.

It is easy to see now that $\Gamma_{\phi,v}\nvdash_{\HfuzzyKbiG}\phi$. Indeed, otherwise, $\Box\Gamma_{\phi,v}\vdash_{\HfuzzyKbiG}\Box\phi$ (using $\mathbf{nec}$ and~$\mathbf{K}$ for $\Box$), whence $\Box\Gamma_{\phi,v},\Th(\HfuzzyKbiG)\vdash_{\HGtriangle}\Box\phi$ by~\eqref{equ:propositionalreduction} (and by completeness of $\HGtriangle$, we have $\Box\Gamma_{\phi,v},\Th(\HfuzzyKbiG)\models_{\biG}\Box\phi$) which leads to a contradiction since $v[\Box\Gamma]>\alpha$, $v[\Th(\HfuzzyKbiG)]=1$ (recall that $\HfuzzyKbiG$ theorems are closed under $\triangle$nec) but $v(\Box\phi)=\alpha$. This means that there is a state $u\in W^\tau$ s.t.\ $1=u[\Th(\HfuzzyKbiG)]\geq u[\Gamma]>u(\phi)$.

Thus, the following relations between $v$ and $u$ hold. Observe that we cannot simply evaluate the assumptions of $\Gamma_{\phi,v}\nvdash_{\HfuzzyKbiG}\phi$ at $1$ (only theorems are guaranteed to be evaluated at $1$ since they are closed under $\triangle$-necessitation). Hence, some conditions are weaker than the original conditions $\#1$--$\#4$ in~\cite{CaicedoRodriguez2015} (the weakenings are \underline{underlined}).
\begin{description}
\item[$\#1$\label{item:no1}] If $v(\Box\theta)>\alpha$, \underline{then $u(\theta)>u(\phi)$}.
\item[$\#2$\label{item:no2}] If $v(\lozenge\theta_1)\leq v(\Box\theta_2)$, then $u(\theta_1)\leq u(\theta_2)$ \underline{or $u(\theta_1)>u(\theta_2)>u(\phi)$}.
\item[$\#3$\label{item:no3}] If $v(\lozenge\theta_1)<v(\Box\theta_2)$, then $u(\theta_1)>u(\theta_2)$, or $u(\theta_1)=u(\theta_2)=1$, \underline{or $u(\theta_2)\leq u(\theta_1)>u(\phi)$}.
\item[$\#4$\label{item:no4}] If $v(\Box\theta)>0$, then $u(\theta)>0$.
\end{description}

Consider now the \emph{contrapositions} of \nameref{item:no2} and \nameref{item:no3}.
\begin{description}
\item[$\#2'$\label{item:no2'}] $\Big(u(\theta_1)>u(\theta_2)\Big)~\&~\Big(u(\theta_1)\leq u(\theta_2)\text{ or }u(\theta_2)\leq u(\phi)\Big)\Rightarrow v(\lozenge\theta_1)>v(\Box\theta_2)$.
\item[$\#3'$\label{item:no3'}] $\Big(u(\theta_1)\!\leq\!u(\theta_2)\Big)~\&~\Big(\{u(\theta_1),u(\theta_2)\}\not\subseteq\{1\}\Big)~\&~\Big(u(\theta_2)\!>\!u(\theta_1)\text{ or }u(\theta_1)\leq u(\phi)\Big)\Rightarrow v(\lozenge\theta_1)\geq v(\Box\theta_2)$.
\end{description}
The left-hand side of~\nameref{item:no2'} can be simplified:
\begin{description}
\item[$\#2^\mathsf{c}$\label{item:no2c}] $u(\theta_1)>u(\theta_2)\Rightarrow v(\lozenge\theta_1)>v(\Box\theta_2)$.
\end{description}

Just as in the original proof, we set $\mathsf{B}=\{v(\Box\theta):\theta\in\Sfconst(\tau)\}$ and $u_b=\min\{u(\theta):v(\Box\theta)=b\}$ for every $b\in\mathsf{B}$. From here, we define
\begin{align}
b_0&=\alpha\nonumber\\
b_{i+1}&=\max\{b:b<b_i\text{ and }u_b<u_{b_i}\}
\label{equ:truthlemmaBoxdecreasingsequence}
\end{align}
which forms a~strictly decreasing sequence $\alpha=b_0>b_1>\ldots>b_N=0$. First, observe that since $\mathsf{B}$ is finite, the sequence is finite as well. We prove by contradiction that the sequence stops at $0$. Indeed, if $b_N=v(\Box\phi_N)>0$, then $u_{b_N}=u(\phi_N)>0$ by \nameref{item:no4}. Since $v(\Box\mathbf{0})\leq v(\Box\phi_N)$, we have $u(\mathbf{0})<u(\phi_N)$, whence $v(\Box\mathbf{0})<v(\Box\phi_N)$ (i.e., there is $b_{N+1}<b_N$, a contradiction).

We also pick formulas $\phi_i\in\Sfconst(\tau)$ s.t.\ $v(\Box\phi_i)=b_i$ and $u(\phi_i)=u_{b_i}$ and an $\varepsilon>0$ s.t.\ $\alpha+\varepsilon<1$. From here, we define
\begin{align}
p_0&=(\alpha+\varepsilon)\wedge_{\mathsf{G}}\min\{v(\lozenge\theta):\theta\in\Sfconst(\tau)\text{ and }v(\lozenge\theta)>\alpha\}\nonumber\\
p_{i+1}&=b_i\wedge_\mathsf{G}\min\{v(\lozenge\theta):\theta\in\Sfconst(\tau)\text{ and }v(\lozenge\theta)>b_{i+1}\}\label{equ:truthlemmaBoxpi}
\end{align}
We take a strictly increasing function $g:[0,1]\rightarrow[0,1]$ s.t.
\begin{align}
g(1)&=1\nonumber\\
g[[u_\alpha,1)]&=[\alpha,p_0)\nonumber\\
g[[u_{b_{i+1}},u_{b_i})]&=[b_{i+1},p_{i+1})
\label{equ:truthlemmaBoxg}
\end{align}
It is clear that $w=g\circ u$ is a state in the canonical model and that $w(\phi)<p_0\leq\alpha+\varepsilon$. Thus, it remains to show that $v\mathsf{R}^\tau w>w(\phi)$ for every $\theta\in\Sfconst$.

\fbox{Case 1.} If $u(\theta)=1$, then $w(\theta)=1$, whence, $v(\Box\theta)\leq w(\theta)$. Now, assume for contradiction that $v(\lozenge\theta)\leq\alpha=v(\Box\phi)$. Then we have from \nameref{item:no2} that either (i) $u(\theta)\leq u(\phi)<1$ or (ii) $u(\theta)>u(\phi)>u(\phi)$. In the first case, we have a contradiction since $u(\theta)=1$ because $w(\theta)=1$. In the second case, the contradiction is immediate. Thus, $v(\lozenge\theta)\geq p_0$.

\fbox{Case 2.} Now, we have two options. Either (i) $u(\theta)\in[u_{b_i},u_{b_{i-1}})$ or (ii) $u(\theta)\in[u_\alpha,1)$. In both cases, we have that $w(\theta)\in[b_i,p_i)$. In (i), we obtain $b_i=\max\{v(\Box\psi):u(\psi)<u_{b_{i-1}}\}$, whence $v(\Box\theta)\leq b_i\leq w(\theta)$. In (ii), we have that $w(\theta)\geq b_0=\alpha$. If $u(\theta)=u_\alpha$, then $v(\Box\theta)\leq\alpha$ and $w(\theta)=\alpha$ from \nameref{item:no1}. And if $u(\theta)>u_\alpha$, then $w(\theta)>w(\alpha)$ (i.e., $v(\Box\theta)\rightarrow_{\mathsf{G}}w(\theta)>w(\alpha)$).

Moreover, if $u(\theta)=u_{b_i}=u(\phi_i)$, then since $u_{b_i}<1$ and $u(\theta)\leq u(\phi)$, we have $w(\theta)=b_i=v(\Box\theta)\leq v(\lozenge\theta)$ by \nameref{item:no3'}. Finally, if $u(\theta)>u_{b_i}$, we have that $u(\theta)>u(\phi_i)\leq u(\phi)$, whence $v(\lozenge\theta)>v(\Box\phi_i)=b_i$ by \nameref{item:no2c}. Thus, $w(\theta)<p_i\leq v(\lozenge\theta)$.

It is clear that in both cases,
\begin{align*}
\inf\{v(\Box\theta)\rightarrow_\mathsf{G}w(\theta):\theta\in\Sfconst(\tau)\}>w(\phi)&\text{ and }\inf\{w(\theta)\rightarrow_\mathsf{G}v(\lozenge\theta):\theta\in\Sfconst(\tau)\}\geq p_0>\alpha
\end{align*}
The result follows.
\end{proof}
\begin{remark}\label{rem:whatisnewBox}
Let us return once again to the proof of Lemma~\ref{lemma:truthlemmaBox} and summarise its differences from~\cite[Claim~1]{CaicedoRodriguez2015}. First, our conditions \nameref{item:no1}--\nameref{item:no4} are weaker than the original ones. Second, since \nameref{item:no1} states that $u(\theta)>u(\phi)$ for $v(\Box\theta)>\alpha$ (and not $u(\theta)=1$), we can obtain only $\inf\{v(\Box\theta)\rightarrow_\mathsf{G}w(\theta):\theta\in\Sfconst(\tau)\}>w(\phi)$ and not $\inf\{v(\Box\theta)\rightarrow_\mathsf{G}w(\theta):\theta\in\Sfconst(\tau)\}=1$ as originally.
\end{remark}
\begin{lemma}\label{lemma:truthlemmaDiamond}
Let $\alpha$, $\varepsilon_1$, and $\varepsilon_2$ be positive and let further $v(\lozenge\phi)=\alpha$. Then, there is $w$ s.t.\ $w(\alpha)\geq\alpha-\varepsilon_1$ and $v\mathsf{R}^\tau w\geq\alpha-\varepsilon_2$.
\end{lemma}
\begin{proof}
Again, the proof follows that of~\cite[Claim~2]{CaicedoRodriguez2015}. We code the minimal requirements for $w$ using a~finite set of formulas $\Xi_{\phi,v}$, to obtain $u\in W$ satisfying those requirements, and then transform $u$ into the required $w$ constructing a suitable map of $[0,1]$ into itself.

We define
\begin{align}
\Xi_1&=\{\theta:v(\theta)<\alpha\text{ and }\theta\in\Sfconst(\tau)\}\nonumber\\
\Xi_2&=\{\theta_2\rightarrow\theta_1:\alpha<v(\lozenge\theta_1)<v(\Box\theta_2)\text{ and }\theta_1,\theta_2\in\Sfconst(\tau)\}\nonumber\\
\Xi_3&=\{(\theta_1\rightarrow\theta_2)\rightarrow\theta_1:v(\lozenge\theta_1)=v(\Box\theta_2)<\alpha\text{ and }\theta_1,\theta_2\in\Sfconst(\tau)\}\nonumber\\
\Xi_{\phi,v}&=\Xi_1\cup\Xi_2\cup\Xi_3
\label{equ:truthlemmaDiamondrefutingtheory}
\end{align}
It is clear that $v(\lozenge\xi)<\alpha$ for every $\xi\in\Xi_{\phi,v}$ and that $\Xi_{\phi,v}\neq\varnothing$ is finite. Thus, we have that $\phi\nvdash_{\HfuzzyKbiG}\bigvee\limits_{\xi\in\Xi}\xi$ via an application of~\eqref{equ:propositionalreduction}, \eqref{equ:HfuzzyKbiGdeductiontheorem}, $\mathbf{K}$, and $\mathbf{nec}$ (this time for $\lozenge$). Again, the argument here is the same as in~\cite[Claim~2]{CaicedoRodriguez2015}, so we omit it for the sake of brevity. Thus, there is a~valuation $u$~s.t.\ $1\!\geq\!u(\phi)\!>\!u[\Xi_{\phi,v}]$ that evaluates all theorems at $1$ (observe that we cannot evaluate $\phi$ at $1$ by default; theorems, on the other hand, are closed under $\triangle$ necessitation and thus are evaluated at~$1$). In addition, the following statements hold w.r.t.\ $u$ (observe that in contrast to Lemma~\ref{lemma:truthlemmaBox}, we could preserve the original conditions on $u$ from~\cite[Claim~2]{CaicedoRodriguez2015}).
\begin{description}
\item[$\#\#1$\label{item:nono1}] If $v(\lozenge\theta)<\alpha$, then $u(\theta)<u(\phi)\leq1$.
\item[$\#\#2$\label{item:nono2}] If $\alpha>v(\lozenge\theta_1)<v(\Box\theta_2)$, then $u(\theta_1)<u(\theta_2)$.
\item[$\#\#3$\label{item:nono3}] If $\alpha>v(\lozenge\theta_1)\leq v(\Box\theta_2)$, then $u(\theta_1)\leq u(\theta_2)$.
\item[$\#\#4$\label{item:nono4}] If $u(\theta)=0$, then $v(\Box\theta)=0$.
\item[$\#\#5$\label{item:nono5}] If $v(\lozenge\theta)=0$, then $u(\theta)=0$. 
\end{description}

We now proceed in a manner dual to that of Lemma~\ref{lemma:truthlemmaDiamond}. We define $\mathsf{C}=\{c:v(\lozenge\theta)\leq\alpha\text{ and }\theta\in\Sfconst(\tau)\}$ and set $u_c=\max\{u(\theta):v(\lozenge\theta)=c\}$ for every $c\in\mathsf{C}$. It is clear that $u_0=0$ (by~\nameref{item:nono5}) and $u_\alpha=\max\{u_c:c\in\mathsf{C}\}$. Now, define the following strictly increasing sequence
\begin{align}
c_0&=v(\lozenge\mathbf{0})=0\nonumber\\
c_{i+1}&=\min\{c\in\mathsf{C}:c>c_i\text{ and }u_c>u_{c_i}\}
\end{align}
We choose $\phi_i$'s s.t.\ $u_{c_i}=u(\phi_i)$ and $c_i=v(\lozenge\phi_i)$. It is clear that the sequence is finite since $\mathsf{C}$ is finite and that it stops at some $c_N=\alpha$. Indeed, if $c_i=v(\lozenge\phi_i)<\alpha$, then $u_{c_i}(\phi_i)<u(\phi)\leq u_\alpha\leq1$ from \nameref{item:nono1} which implies the existence of $c_{i+1}$.

Now fix $\varepsilon>0$ s.t.\ $\alpha-\varepsilon>c_{N-1}$ and define
\begin{align}
q_{N-1}&=(\alpha-\varepsilon)\vee_\mathsf{G}\max\{v(\Box\theta):v(\Box\theta)<c_N\}\nonumber\\
q_i&=c_i\vee_\mathsf{G}\max\{v(\Box\theta):v(\Box\theta)<c_{i+1}\}
\end{align}
This gives us two sequences:
\begin{align*}
0=c_0\leq q_0<c_1\leq q_1<\ldots<c_{N-1}\leq\alpha-\varepsilon\leq q_{N-1}<c_N=\alpha&\text{ and }
0=u_{c_0}<u_{c_1}<\ldots<u_{c_N}=u_\alpha
\end{align*}

We now have two cases: (1) $u(\phi)<1$ and (2) $u(\phi)=1$. In the first case, we fix another $\varepsilon'>0$ s.t.\ $\alpha-\varepsilon\leq q_{N-1}<\alpha-\varepsilon'<c_N=\alpha$ and choose a~strictly increasing function $g:[0,1]\rightarrow[0,1]$ s.t.
\begin{align}
g(0)&=0\nonumber\\
g[(u_{c_i},u_{c_{i+1}}]]&=(q_i,c_{i+1}]\quad(i<N-2)\nonumber\\
q[(u_{c_{N-2}},u_{c_{N-1}})]&=(q_{N-1},\alpha-\varepsilon')\nonumber\\
g(u(\phi))&=\alpha-\varepsilon'\nonumber\\
g[(u(\phi),u_\alpha]]&=\left(\alpha-\varepsilon',\frac{2\alpha-\varepsilon'}{2}\right]\nonumber\\
g[(u_\alpha,1)]&=\left(\frac{2\alpha-\varepsilon'}{2},1\right)\nonumber\\
g(1)&=1
\label{equ:truthlemmaDiamondg}
\end{align}
Note that we cannot always send $u_\alpha$ to $\alpha$ since it is possible that $v(\lozenge\phi)=1$ but $u(\phi)\neq1$ always (if, for example, $\phi=p\wedge{\sim}\triangle p$). Now, we can set $w=g\circ u$. It is clear that $w(\phi)\geq\alpha-\varepsilon'$. It remains only to show that $v\mathsf{R}^\tau w\geq\alpha-\varepsilon$. This is done exactly as in~\cite[Claim~2]{CaicedoRodriguez2015}.

\fbox{Case 1.1} If $v(\lozenge\theta)\geq\alpha$, then, trivially, $(w(\theta)\rightarrow_\mathsf{G}v(\lozenge\theta))\geq\alpha>\alpha-\varepsilon$.

\fbox{Case 1.2} If $v(\lozenge\theta)<\alpha$, then $u(\theta)<u(\phi)\leq1$ from \nameref{item:nono1} and we have two cases: (i) $u(\theta)\in(u_{c_i},u_{c_{i+1}}]$ or (ii) $u(\theta)=0$. In the first case, $w(\theta)\in(q_i,c_{i+1}]$, and, moreover, $c_{i+1}=v(\lozenge\phi_{i+1})=\min\{v(\lozenge\psi):u(\psi)>u_{c_i}\}$. Hence, $v(\lozenge\theta)\geq c_{i+1}\geq w(\theta)$. In the second case, using \nameref{item:nono4}, we have $w(\theta)=0$, whence $v(\Box\theta)=0$.

\fbox{Case 1.3} If $v(\Box\theta)<\alpha$, then $v(\Box\theta)>c_{N-1}=v(\lozenge\phi_{N-1})$, whence $u(\theta)>u(\phi_{N-1})$ (from \nameref{item:nono2}) and thus, $w(\theta)>q_{N-1}>\alpha-\varepsilon$ (from~\eqref{equ:truthlemmaDiamondg}). Hence, $(v(\Box\theta)\rightarrow_\mathsf{G}w(\theta))>\alpha-\varepsilon$.

\fbox{Case 1.4} If $v(\Box\theta)<\alpha$, then $c_i\leq v(\Box\theta)\leq q_i<c_{i+1}$ and we have two options: (i) $v(\Box\theta)=c_i$ and (ii) $v(\Box\theta)>c_i$. If (i), we have $v(\Box\theta)=v(\lozenge\phi_i)$, whence $u_{c_i}=u(\phi_i)\leq u(\theta)$ using \nameref{item:nono3}, and thus, $c_i\leq w(\theta)$ using~\eqref{equ:truthlemmaDiamondg}. Hence, $v(\Box\theta)\leq w(\theta)$. If (ii), then $u_{c_i}<u(\theta)$ using \nameref{item:nono2}, whence, $q_i\leq w(\theta)$ from~\eqref{equ:truthlemmaDiamondg} and thus, again $v(\Box\theta)\leq w(\theta)$.

Now we can see that cases 1.1--1.4 imply
\begin{align*}
\inf\{w(\theta)\rightarrow_\mathsf{G}v(\lozenge\theta):\theta\in\Sfconst(\tau)\}\geq\alpha&\text{ and }\inf\{v(\Box\theta)\rightarrow_\mathsf{G}w(\theta)\}>\alpha-\varepsilon
\end{align*}

In the second case, we define $g$ exactly as in~\cite[Claim~2]{CaicedoRodriguez2015}:
\begin{align}
g(0)&=0\nonumber\\
g[(u_{c_i},u_{c_{i+1}}]]&=(q_i,c_{i+1}]\quad(i<N-2)\nonumber\\
g[(u_{c_{N-1}},1)]&=(q_{N-1},\alpha)\nonumber\\
g(1)&=1
\label{equ:truthlemmaDiamondg1}
\end{align}
In this case, we have that $w(\phi)=1\geqslant\alpha-\varepsilon'$ for any $\varepsilon'>0$. We can also show that $v\mathsf{R}^\tau w\geq\alpha-\varepsilon$ in the same way as in the first case. The result follows.
\end{proof}
\begin{remark}\label{rem:whatisnewDiamond}
Again, let us survey the differences between our proof and~\cite[Claim~2]{CaicedoRodriguez2015}. First, it is not always the case that we can construct $w$ s.t.\ $w(\phi)=1$. This might fail even if $v(\lozenge\phi)=1$. This is why, $u_\alpha$ is only the greatest among $u_c$'s but not necessarily $1$ and is not even necessarily equal to $u(\phi)$. Furthermore, we cannot always send $u(\phi)$ (or $u_\alpha$) back to $\alpha$ itself and have to consider the case where we find an arbitrarily small $\varepsilon'$ s.t.\ $q_{N-1}<g(u(\phi))=\alpha-\varepsilon'\leq g(u_\alpha)<\alpha$ as shown in~\eqref{equ:truthlemmaDiamondg}. This makes $g$ more fine-grained than originally.
\end{remark}

We are now ready to prove the truth lemma.
\begin{lemma}\label{lemma:truthlemma}
For every $\phi\in\Sfconst(\tau)$ and $u\in W^\tau$, it holds that $u(\phi)=e^\tau(\phi,u)$.
\end{lemma}
\begin{proof}
We proceed by induction on $\phi$. The basis case of propositional variables holds by Definition~\ref{def:canonicalmodel}; the cases of propositional connectives can be obtained via a straightforward application of the induction hypotheses. Finally, let us consider the modal cases.

If $\phi=\Box\psi$, we have two options. First, $u(\Box\psi)=1$. By Definition~\ref{def:canonicalmodel}, $u\mathsf{R}^\tau u'\leq u(\Box\psi)\rightarrow_\mathsf{G}u'(\psi)$, whence $u(\Box\psi)\leq\inf\{u(\Box\psi)\rightarrow_\mathsf{G}u'(\psi):u'\in W^\tau\}$. It is now immediate that $u(\Box\psi)=\inf\{u(\Box\psi)\rightarrow_\mathsf{G}u'(\psi):u'\in W^\tau\}$. Otherwise, if $u(\Box\psi)<1$, we obtain the result by Lemma~\ref{lemma:truthlemmaBox}.

If $\phi=\lozenge\psi$, we proceed in the dual manner: if $u(\lozenge\psi)=0$, the result is immediate since $u\mathsf{R}^\tau u'\leq u'(\psi)\rightarrow_\mathsf{G}u(\lozenge\psi)$. If $u(\lozenge\psi)>0$, we use Lemma~\ref{lemma:truthlemmaDiamond}.

The result follows.
\end{proof}
\begin{theorem}\label{theorem:HfuzzyKbiGweakcompleteness}
Let $\Gamma\cup\{\phi\}\subsetneq\bimodalLtriangle$ be finite. Then $\Gamma\models_{\fuzzyKbiG}\phi$ iff $\Gamma\vdash_{\HfuzzyKbiG}\phi$.
\end{theorem}
\begin{proof}
The soundness can be obtained by a routine check of axioms' and rules' validity. The completeness follows from Lemma~\ref{lemma:truthlemma} and~\eqref{equ:HfuzzyKbiGdeductiontheorem}.
\end{proof}

We finish the section by establishing the strong completeness result. We adapt the proof from~\cite{CaicedoRodriguez2015} in the same manner that we did for the crisp $\KbiG$ in~\cite[Theorem~3.2]{BilkovaFrittellaKozhemiachenko2023IGPL}.
\begin{theorem}\label{theorem:KbiGstrongcompleteness}
$\HfuzzyKbiG$ is strongly complete: for any $\Gamma\cup\{\phi\}\subseteq\bimodalLtriangle$, it holds that $\Gamma\models_{\fuzzyKbiG}\phi$ iff $\Gamma\vdash_{\HfuzzyKbiG}\phi$.
\end{theorem}
\begin{proof}
The proof follows~\cite[Theorem~3.1]{CaicedoRodriguez2015}. The only two differences are that we need to account for $\triangle$ and that the $\fuzzyKbiG$ entailment $\Gamma\models_{\fuzzyKbiG}\chi$ is defined via the order on $[0,1]$. That is, if the entailment is refuted by $e$, then $\inf\{e(\phi,c):\phi\in\Gamma\}>e(\chi,c)$ for some $c\in\mathfrak{F}$. This, in turn, is equivalent to
\begin{align}
\exists d\!\in\!(0,1]~\forall\phi\in\Gamma:e(\phi,c)\geq d\text{ but }e(\chi,c)<d\label{equ:entailmentrefutation}
\end{align}

Now let $\Gamma\cup\{\phi\}\subseteq\bimodalLtriangle$ and $\Gamma\nvdash_{\HcrispKbiG}\phi$. We consider the classical first order theory $\Gamma^*$ whose signature contains two unary predicates $W$ and $P$, one binary predicate $<$, binary functions $\circ$ and $\mathsf{s}$, unary function $\blacktriangle$, constants $0$, $1$, $c$, $d$, and a function symbol $f_\theta$ for each $\theta\in\bimodalLtriangle$. Intuitively, $W(x)$ stands for ‘$x$ is a state’; $P(x)$ for ‘$x$ is a number’; $f_\theta(x)$ for ‘the value of $\theta$ in $x$’; $<$ is going to be the order on numbers. Constants $c$ and $d$ stand for the state where the entailment is refuted and the value that separates $\Gamma$ and $\chi$~--- cf.~\eqref{equ:entailmentrefutation}, respectively. $\circ$ is used to define the value of the G\"{o}del implication; $\blacktriangle$ is the counterpart of $\triangle$; $\mathsf{s}$ is the relation between states.

We can now formalise the semantics of $\fuzzyKbiG$ in the classical first-order logic as follows.\footnote{Note that in previous proofs of strong completeness of $\KG$ and $\crispKbiG$ in~\cite{CaicedoRodriguez2015,RodriguezVidal2021,BilkovaFrittellaKozhemiachenko2023IGPL}, $\Gamma^*$ included $\forall x(W(x)\vee{\sim}W(x))$ (interpreted as ‘every $x$ is a state or not a~state’) as an axiom. This formula is classically valid and thus superfluous, whence, we do not add it here.}
\begin{itemize}
\item $\forall x{\sim}(W(x)\wedge P(x))$
\item $P(d)$
\item ‘$\langle P,<\rangle$ is a strict linear order s.t.\ $0<d\leq1$, $0$ and $1$ are the minimum and the maximum of $\langle P,<\rangle$’.
\item $\forall x\forall y((W(x)\wedge W(y))\rightarrow P(\mathsf{s}(x,y)))$
\item $\forall x\forall y((P(x)\wedge P(y))\rightarrow((x\leq y\wedge x\circ y=1)\vee(x>y\wedge x\circ y=y)))$
\item $\forall x(P(x)\rightarrow((x=1\wedge\blacktriangle(x)=1)\vee(x<1\wedge\blacktriangle(x)=0)))$
\item For each $\theta,\theta'\in\bimodalLtriangle$, we add the following formulas.
\begin{itemize}
\item $\forall x(W(x)\rightarrow P(f_\theta(x)))$
\item $\forall x(W(x)\rightarrow f_{{\sim}\theta}(x)=(f_\theta(x)\circ 0))$
\item $\forall x(W(x)\rightarrow f_{\triangle\theta}(x)=\blacktriangle
(f_\theta(x)))$
\item $\forall x(W(x)\rightarrow f_{\theta\wedge\theta'}(x)=\min\{
f_\theta(x),f_{\theta'}(x)\})$
\item $\forall x(W(x)\rightarrow f_{\theta\vee\theta'}(x)=\max\{
f_\theta(x),f_{\theta'}(x)\})$
\item $\forall x(W(x)\rightarrow f_{\theta\rightarrow\theta'}(x)=
f_\theta(x)\circ f_{\theta'}(x))$
\item $\forall x(W(x)\rightarrow f_{\Box\theta}(x)=\inf\limits_{y}\{\mathsf{s}(x,y)\circ f_\theta(y)\})$
\item $\forall x(W(x)\rightarrow f_{\lozenge\theta}(x)=\sup\limits_{y}\{\min\{\mathsf{s}(x,y),f_\theta(y)\}\})$
\end{itemize}
\item For each $\gamma\in\Gamma$, we add $f_\gamma(c)\geq d$.
\item We also add $W(c)\wedge(f_\phi(c)<d)$.
\end{itemize}

The rest of the proof is identical to that in~\cite{CaicedoRodriguez2015}. For each finite subset $\Gamma^-$ of $\Gamma^*$, we let $\bimodalLtriangle^-=\{\theta:f_\theta\text{ occurs in }\Gamma^-\}$. Since $\bimodalLtriangle^-\cap\Gamma\nvdash_{\HKbiG}\phi$ by assumption, Theorem~\ref{theorem:HfuzzyKbiGweakcompleteness} entails that there is a~crisp pointed model $\langle\mathfrak{M},c\rangle$ with $\mathfrak{M}=\langle W,\mathsf{s}^{\Gamma^-},e^{\Gamma^-}\rangle$ being such that $e^{\Gamma^-}(\phi,c)<d$ and $e^{\Gamma^-}(\theta,c)\geq d$ for every $\theta\in\Gamma\cap\Gamma^-$. Thus, the following structure
\[\langle W\uplus[0,1],W,[0,1],<,0,1,c,d,\circ,\blacktriangle,\mathsf{s}^{\Gamma^-},\{f_\theta\}_{\theta\in\bimodalLtriangle}\rangle\]
is a model of $\Gamma^-$. Now, by compactness and the downward L\"{o}wenheim--Skolem theorem, $\Gamma^*$ has a~countable model
\[\mathfrak{M}^*=\langle B,W,P,<,0,1,c,d,\circ,\blacktriangle,\mathsf{s}\{f_\theta\}_{\theta\in\bimodalLtriangle}\rangle\]
Now, we can embed $\langle P,<\rangle$ into $\langle\mathbb{Q}\cap[0,1],<\rangle$ preserving $0$ and $1$ as well as all infima and suprema. Hence, we may w.l.o.g.\ assume that $\mathsf{s}$ is crisp and the ranges of $f_\theta$'s are contained in $[0,1]$. Then, it is straightforward to verify that $\mathfrak{M}=\langle W,S,e\rangle$, where $e(\theta,w)=f_\theta(w)$ for all $w\in W$ and $\theta\in\bimodalLtriangle$, is a crisp $\KbiG$ model with a distinguished world $c$ such that $v[\Gamma,c]\geq d$ and $e(\phi,c)<d$ for some $0<d\leq1$. Hence, $\inf\{e(\gamma,c):\gamma\in\Gamma\}>e(\phi,c)$, and thus, $\Gamma\not\models_{\fuzzyKbiG}\phi$.
\end{proof}
\section[Paraconsistent relatives]{$\KGsquare$ and $\infoGsquare$ --- paraconsistent relatives of $\KbiG$\label{sec:paraconsistentlogics}}
As we have mentioned in the introduction, we will be mainly concerned with paraconsistent relatives of $\KbiG$ defined on \emph{bi-relational} frames of the form $\mathfrak{F}=\langle W,R^+,R^-\rangle$ where $R^+$ and $R^-$ are (possibly) fuzzy relations on $W$. Let us quickly recall the motivation for the bi-relational frames and informational modalities that was outlined in~\cite{BilkovaFrittellaKozhemiachenko2023nonstandard}.

We begin with the interpretations of modalities. If we interpret the states in a~Kripke frame as sources that refer to one another and the accessibility relations as the degree of trust, then all modalities correspond to different strategies of aggregating the information.

Namely, $\Box$ and $\lozenge$ stand for the ‘pessimistic’ and ‘optimistic’ aggregations: the positive support of $\Box\phi$ is calculated using the infima of the positive supports of $\phi$ and thus, $e_1(\Box\phi,w)<1$ as long as there exists $w'$ that does not completely positively support $\phi$ (i.e., $e_1(\phi,w')<1$) and is trustworthy enough for $w$ (i.e., $wR^+w'>e_1(\phi,w')$). The negative support of $\Box\phi$ is calculated using the suprema of the negative supports of $\phi$, i.e., it suffices to find a sufficiently trusted source $w''$ with $wR^-w''>0$ that gives $\phi$ some non-zero negative support ($e_2(\phi,w'')>0$) for $e_2(\Box\phi,w)>0$. Thus, $\Box\phi$ represents the search for trustworthy refutations of $\phi$ (and only if they are not found can $\Box\phi$ be evaluated at $(1,0)$). Dually, since $\lozenge\phi$ uses suprema of positive supports and infima of negative supports, it can be seen as the search of trusted confirmations of $\phi$ (and if these are not found $\lozenge\phi$ is evaluated at $(0,1)$).

The informational modalities stand for the ‘sceptical’ ($\blacksquare$) and ‘credulous’ ($\blacklozenge$) aggregations. Their support of truth is defined in the same manner as that of $\Box$ and $\lozenge$. The support of falsity, however, is the same as the support of truth: $e_2(\blacksquare\phi,w)$ is calculated using the \emph{infima} of $e_2(\phi,w')$ and $e_2(\blacklozenge\phi,w)$ uses the \emph{suprema}. Here $w$ either looks for \emph{trusted rejections}\footnote{We differentiate between a~\emph{rejection} which we treat as \emph{lack of support} and a~\emph{denial, disproof, refutation, counterexample}, etc.\ which we interpret as the \emph{negative support}.} (using $\blacksquare$) or (using $\blacklozenge$) for \emph{trusted confirmations of both positive and negative supports} of $\phi$.

Observe that the independence of the support of truth from the support of falsity is crucial to differentiate sceptical (credulous) and pessimistic (optimistic) aggregations. Indeed, if we did not treat them separately, sceptical and pessimistic (and, likewise, credulous and optimistic) aggregations would coincide. 

In what follows, we use $R^+$ to compute the support of the truth of the modal formulas and $R^-$ for the support of falsity. As modalities represent aggregation from different sources, it is reasonable to assume that one can trust a confirmation from a source more or less than a denial. For example, if we read a sensationalistic newspaper, we might be less inclined to believe its assertions than refutations; likewise, if we are listening to an extremely sceptical person, we might believe their confirmations more than their denials.
\subsection{Semantics\label{ssec:paraconsistentsemantics}}
The languages $\bimodalLtrianglesquare$ and $\infobimodalLtrianglesquare$ of $\KGsquare$ and $\infoGsquare$ are defined via the following grammars.
\begin{align*}
\bimodalLtrianglesquare\ni\phi&\coloneqq p\in\Prop\mid\neg\phi\mid{\sim}\phi\mid\triangle\phi\mid(\phi\wedge\phi)\mid(\phi\vee\phi)\mid(\phi\rightarrow\phi)\mid\Box\phi\mid\lozenge\phi\\
\infobimodalLtrianglesquare\ni\phi&\coloneqq p\in\Prop\mid\neg\phi\mid{\sim}\phi\mid\triangle\phi\mid(\phi\wedge\phi)\mid(\phi\vee\phi)\mid(\phi\rightarrow\phi)\mid\blacksquare\phi\mid\blacklozenge\phi
\end{align*}

Since the languages differ only with their modal operators and are interpreted on the same classes of frames, we put their semantics in one definition below.
\begin{definition}[Semantics of $\KGsquare$ and $\infoGsquare$]\label{def:paraconsistentsemantics}
A bi-relational paraconsistent model is a tuple $\mathfrak{M}\!=\!\langle W,R^+,R^-,e_1,e_2\rangle$ with $\langle W,R^+,R^-\rangle$ being a bi-relational frame and $e_1,e_2:\Prop\times W\rightarrow[0,1]$.

The valuations are extended on complex propositional formulas as follows.
\begin{longtable}{rclrcl}
$e_1(\neg\phi,w)$&$=$&$e_2(\phi,w)$&$e_2(\neg\phi,w)$&$=$&$e_1(\phi,w)$\\
$e_1(\phi\wedge\phi',w)$&$=$&$e_1(\phi,w)\wedge_\mathsf{G}e_1(\phi',w)$&$e_2(\phi\wedge\phi',w)$&$=$&$e_2(\phi,w)\vee_\mathsf{G}e_2(\phi',w)$\\
$e_1(\phi\vee\phi',w)$&$=$&$e_1(\phi,w)\vee_\mathsf{G}e_1(\phi',w)$&$e_2(\phi\vee\phi',w)$&$=$&$e_2(\phi,w)\wedge_\mathsf{G}e_2(\phi',w)$\\
$e_1(\phi\rightarrow\phi',w)$&$=$&$e_1(\phi,w)\!\rightarrow_\mathsf{G}\!e_1(\phi',w)$&$e_2(\phi\rightarrow\phi',w)$&$=$&$e_2(\phi',w)\coimplies_\mathsf{G}e_2(\phi,w)$\\
$e_1({\sim}\phi,w)$&$=$&${\sim_\mathsf{G}}e_1(\phi,w)$&$e_2({\sim}\phi,w)$&$=$&$1\coimplies_\mathsf{G}e_2(\phi,w)$\\
$e_1(\triangle\phi,w)$&$=$&$\triangle_\mathsf{G}e_1(\phi,w)$&$e_2(\triangle\phi,w)$&$=$&${\sim_\mathsf{G}\sim_\mathsf{G}}e_2(\phi,w)$
\end{longtable}

The modal formulas are interpreted as follows.
\begin{center}
\begin{tabular}{rclrcl}
$e_1(\Box\phi,w)$&$=$&$\inf\limits_{w'\in W}\!\{wR^+w'\rightarrow_\mathsf{G}e_1(\phi,w')\}$
&
$e_2(\Box\phi,w)$&$=$&$\sup\limits_{w'\in W}\!\{wR^-w'\wedge_\mathsf{G}e_2(\phi,w')\}$\\
$e_1(\lozenge\phi,w)$&$=$&$\sup\limits_{w'\in W}\!\{wR^+w'\wedge_\mathsf{G}e_1(\phi,w')\}$
&
$e_2(\lozenge\phi,w)$&$=$&$\inf\limits_{w'\in W}\!\{wR^-w'\rightarrow_\mathsf{G}e_2(\phi,w')\}$\\
$e_1(\blacksquare\phi,w)$&$=$&$\inf\limits_{w'\in W}\!\{wR^+w'\!\!\rightarrow_\mathsf{G}e_1(\phi,w')\}$
&
$e_2(\blacksquare\phi,w)$&$=$&$\inf\limits_{w'\in W}\!\{wR^-w'\!\!\rightarrow_\mathsf{G}e_2(\phi,w')\}$\\
$e_1(\blacklozenge\phi,w)$&$=$&$\sup\limits_{w'\in W}\!\{wR^+w'\!\wedge_\mathsf{G}\!e_1(\phi,w')\}$
&
$e_2(\blacklozenge\phi,w)$&$=$&$\sup\limits_{w'\in W}\!\{wR^-w'\!\wedge_\mathsf{G}\!e_2(\phi,w')\}$
\end{tabular}
\end{center}

We say that $\phi$ is \emph{$e_1$-valid ($e_2$-valid) on a~pointed frame $\langle\mathfrak{F},w\rangle$} ($\mathfrak{F},w\models^+\phi$ and $\mathfrak{F},w\models^-\phi$, respectively) iff $e_1(\phi,w)=1$ ($e_2(\phi,w)=0$) for every model $\mathfrak{M}$ on $\mathfrak{F}$. We call $\phi$ \emph{strongly valid on $\langle\mathfrak{F},w\rangle$} ($\mathfrak{F},w\models\phi$) iff it is $e_1$ and $e_2$-valid. We call $\phi$ $e_1$-valid (respectively, $e_2$-valid, strongly valid) on $\mathfrak{F}$, iff $\mathfrak{F},w\models^+\phi$ ($\mathfrak{F},w\models^-\phi$, $\mathfrak{F},w\models\phi$, respectively) for every $w\in\mathfrak{F}$.

$\Gamma$ \emph{entails} $\chi$ (on $\mathfrak{F}$) iff for every model $\mathfrak{M}$ (on $\mathfrak{F}$) and every $w\in\mathfrak{M}$, it holds that
\begin{align*}
\inf\{e_1(\phi,w):\phi\in\Gamma\}\leq e_1(\chi,w)\text{ and }\sup\{e_2(\phi,w):\phi\in\Gamma\}\geq e_2(\chi,w)
\end{align*}
\end{definition}

One can see from Definition~\ref{def:paraconsistentsemantics} that \emph{the propositional} fragment of both logics is, in fact, $\Gsquare$, a~pa\-ra\-con\-sis\-tent expansion of G\"{o}del logic with $\neg$\footnote{Observe that $\phi\coimplies\chi$ can be defined via $\rightarrow$ in the presence of $\neg$ as follows: $\neg(\neg\chi\rightarrow\neg\phi)$; $\triangle\phi$ can be defined as $\neg{\sim\sim}\neg\phi$.} introduced in~\cite{Ferguson2014} as $\mathsf{I_4C_4G}$ and then in~\cite{BilkovaFrittellaKozhemiachenko2021} under its current designation (cf.~Fig.~\ref{fig:logics}). In addition, support of falsity conditions of $\KGsquare$ coincide with the semantics of $\KbiG$. Let us now recall and expand an example from~\cite{BilkovaFrittellaKozhemiachenko2023nonstandard} that illustrates the semantics of modalities.
\begin{example}\label{example:restaurant}
A tourist ($t$) wants to go to a restaurant and asks their two friends ($f_1$ and $f_2$) to describe their impressions regarding the politeness of the staff ($s$) and the quality of the desserts ($d$). Of course, the friends' opinions are not always internally consistent, nor is it always the case that one or the other even noticed whether the staff was polite or was eating desserts. Furthermore, $t$ trusts their friends to different degrees when it comes to their positive and negative opinions.

The first friend says that half of the staff was really nice but the other half is unwelcoming and rude and that the desserts (except for the tiramisu and souffl\'{e}) are tasty. The second friend, unfortunately, did not have the desserts at all. Furthermore, even though, they praised the staff, they also said that the manager was quite obnoxious.

The situation is depicted in Fig.~\ref{fig:restaurant}. Let us now look at how different aggregations work with this information. If the tourist is sceptical w.r.t.\ $s$~and~$d$, they look for \emph{trusted rejections} of both positive and negative supports of $s$~and~$d$. Thus $t$ uses the values of $R^+$ and $R^-$ as thresholds above which the information provided by the source does not count as a~trusted enough rejection. I.e., to accept rejection from a friend, it should be stronger than the degree of trust the tourist gives to the friend. We have that $tR^+f_1>e_1(s,f_1)$ but $tR^+f_2\leq e(s,f_2)$ (Fig.~\ref{fig:restaurant}). Thus, only the account of the first friend counts as a~rejection. In our case, we have $e(\blacksquare s,t)=(0.5,0.5)$ and $e(\blacksquare d,t)=(0,0)$.

On the other hand, if $t$ is credulous, they look for \emph{trusted confirmations} of both positive and negative supports and use $R^+$ and $R^-$ as thresholds up to which they accept the information provided by the source. In particular, we have $e(\blacklozenge s,t)=(0.7,0.2)$ and $e(\blacklozenge d,t)=(0.7,0.3)$.

Similarly, if $t$ is \emph{pessimistic} about the staff, they will use the account of $f_1$ for the trusted \emph{rejection} of the \emph{positive support} of $s$ and for the trusted confirmation of its \emph{negative support} (since $t$ trusts the rejections provided by $f_1$ more than those provided by $f_2$). Thus, $e(\Box s,t)=(0.5,0.5)$. If $t$ is \emph{optimistic} about the desserts, then they will use the account of $f_1$ for the trusted confirmation of the positive support of $d$. However, $f_2$ \emph{completely rejects} the negative support of $d$, and $t$ has positive trust in $f_2$'s rejections ($tR^-f_2>0$). Hence, the result of the optimistic aggregation is as follows: $e(\lozenge d,t)=(0.7,0)$.
\end{example}

\begin{figure}
\[\xymatrix{f_1:\txt{$s=(0.5,0.5)$\\$d=(0.7,0.3)$}~&&~t~\ar[rr]^(.3){(0.7,0.2)}\ar[ll]_(.3){(0.8,0.9)}&&~f_2:\txt{$s=(1,0.4)$\\$d=(0,0)$}}\]
\caption{$(x,y)$ stands for $wR^+w'=x,wR^-w'=y$. $R^+$ (resp., $R^-$) is interpreted as the tourist's threshold of trust in positive (negative) statements by the friends.}
\label{fig:restaurant}
\end{figure}

The main goal of this section is to study $\fuzzybirelKGsquare$ and $\birelinfoGsquare$. We first show that $\Box$ and $\lozenge$ are not interdefinable in $\birelKGsquare$ (in fact, even in $\crispbirelKGsquare$) in contrast to the mono-relational logics.
\begin{proposition}\label{prop:birelKGsquarenondefinability}
$\Box$ and $\lozenge$ are not interdefinable in $\crispbirelKGsquare$ and hence, in $\fuzzybirelKGsquare$.
\end{proposition}
\begin{proof}
Denote with $\mathscr{L}_\Box$ and $\mathscr{L}_\lozenge$ the $\lozenge$-free and $\Box$-free fragments of $\bimodalLtrianglesquare$, respectively. To prove the statement, it suffices to find a pointed model $\langle\mathfrak{M},w\rangle$ s.t.\ there is no $\mathscr{L}_\lozenge$ formula that has the same value at $w$ as $\Box p$ and vice versa.

Consider the model in Fig.~\ref{fig:birelKGsquarenondefinability}. We have $e(\Box p,w_0)=\left(\frac{3}{5},\frac{3}{4}\right)$ and $e(\lozenge p,w_0)=\left(\frac{4}{5},\frac{2}{4}\right)$.
\begin{figure}
\centering
\[\xymatrix{w_1:p=\left(\frac{4}{5},\frac{1}{4}\right)&w_2:p=\left(\frac{2}{5},\frac{3}{4}\right)&w_3:p=\left(\frac{3}{5},\frac{2}{4}\right)\\&w_0:p=(1,0)\ar[u]|{-}\ar[ur]|{\pm}\ar[ul]|{+}&}\]
\caption{All variables have the same values in all states exemplified by $p$.}
\label{fig:birelKGsquarenondefinability}
\end{figure}

It is easy to check that $e(\phi,t)\in\{e(p,t),e(\neg p,t),(1,0),(0,1)\}$ for every $\phi\!\in\!\bimodalLtrianglesquare$ over one variable on the single-point irreflexive frame with a state $t$. Thus, for every $\chi\in\mathscr{L}_\Box$ and every $\psi\!\in\!\mathscr{L}_\lozenge$ it holds that
\begin{align*}
e(\Box\chi,w_0)&\in\left\{(0;1),\left(\frac{3}{5};\frac{3}{4}\right),\left(\frac{1}{4};\frac{3}{5}\right),\left(\frac{3}{4};\frac{3}{5}\right),\left(\frac{3}{5};\frac{1}{4}\right),(1;0)\right\}=X\\
e(\lozenge\psi,w_0)&\in\left\{(0;1),\left(\frac{4}{5};\frac{2}{4}\right),\left(\frac{2}{4};\frac{2}{5}\right),\left(\frac{2}{4};\frac{4}{5}\right),\left(\frac{2}{5};\frac{2}{4}\right),(1;0)\right\}=Y
\end{align*}
Now, let $X^c$ and $Y^c$ be the closures of $X$ and $Y$ under propositional operations. It is clear\footnote{Note that the closure under G\"{o}delian propositional operations of a given set $\{x_1,\ldots,x_n\}\subseteq[0,1]$ can only add $0$ and $1$ to this set but not an additional $x'\notin\{x_1,\ldots,x_n\}$ s.t.\ $0<x'<1$.} that $\left(\frac{3}{5};\frac{3}{4}\right)\notin Y^c$ and $\left(\frac{4}{5};\frac{2}{4}\right)\notin X^c$. It is also easy to verify by induction that for all $\chi'\in\mathscr{L}_\Box$ and $\psi'\in\mathscr{L}_\lozenge$, it holds that $e(\chi',w_0)\in X^c$ and $e(\psi',w_0)\in Y^c$. The result now follows.
\end{proof}

\begin{figure}
\centering
\[\xymatrix{w~\ar[rr]^(.3){R^+=R^-=\frac{1}{2}}&&~w':p=\left(1,\frac{2}{3}\right)}\]
\caption{A $\fuzzyKGsquare$ countermodel for $\lozenge{\sim\sim}p\rightarrow{\sim\sim}\lozenge p$.}
\label{fig:KGbirelnoextension}
\end{figure}
It is also easy to check (Fig.~\ref{fig:KGbirelnoextension}) that $\lozenge{\sim\sim}p\rightarrow{\sim\sim}\lozenge p$ \emph{is not strongly valid} in $\fuzzybirelKGsquare$ (and, in fact, in $\fuzzyKGsquare$ since the countermodel is \emph{monorelational}), even though it is $\fuzzyKbiG$-valid. Likewise, $\Box\mathbf{0}\vee{\sim}\Box\mathbf{0}$ is $\fuzzyKbiG$-valid but not $\KGsquare$-valid. Namely, $e(\Box\mathbf{0}\vee{\sim}\Box\mathbf{0},w)=(1,\frac{1}{2})$ (Fig.~\ref{fig:KGbirelnoextension}).

On the other hand, $\crispbirelKGsquare$ \emph{does extend} $\crispKbiG$.
\begin{lemma}\label{lemma:splitconflation}
Let $\mathfrak{M}=\langle W,R^+,R^-,e_1,e_2\rangle$ be a crisp $\birelKGsquare$ model. We define $$\mathfrak{M}^*=\langle W,(R^+)^*,(R^-)^*,e^*_1,e^*_2\rangle$$ to be as follows: $(R^+)^*=R^-$, $(R^-)^*=R^+$, $e^*_1(p,w)=1-e_2(p,w)$, and $e^*_2(p,w)=1-e_1(p,w)$.

Then, $e(\phi,w)=(x,y)$ iff $e^*(\phi,w)=(1-y,1-x)$.
\end{lemma}
\begin{proof}
We proceed by induction on $\phi$. The basis case of propositional variables holds by the construction of $\mathfrak{M}^*$. The cases of propositional connectives can be shown as in~\cite[Proposition~5]{BilkovaFrittellaKozhemiachenko2021}. We consider the case of $\phi=\Box\psi$ since $\phi=\lozenge\psi$ can be tackled in the same manner.

Let $e(\Box\psi,w)=(x,y)$. Then $\inf\{e_1(\psi,w'):wR^+w'\}=x$, and $\sup\{e_2(\psi,w'):wR^-w'\}=y$. Now, we apply the induction hypothesis to $\psi$, and thus if $e(\psi,s)=(x',y')$, then $e^*_1(\psi,s)=1-y'$ and $e^*_2(\psi,s')=1-x'$ for any $s\in R^+(w)=(R^-)^*(w)$ and $s'\in R^-(w)=(R^+)^*(w)$. But then $\inf\{e^*_1(\psi,w'):w(R^+)^*w'\}=1-y$, and $\sup\{e^*_2(\psi,w'):w(R^-)^*w'\}=1-x$, as required.
\end{proof}

\begin{proposition}\label{prop:birelKGsquarecrispextension}
Let $\phi\in\bimodalLtriangle$. Then, $\phi$ is $\crispKbiG$-valid iff it is $\crispbirelKGsquare$-valid.
\end{proposition}
\begin{proof}
Since the $e_1$-conditions coincide with the semantics of $\KbiG$, it is clear that if $\phi$ is \emph{not} $\KbiG$ valid, then it is not $\birelKGsquare$-valid either. For the converse, it follows from Lemma~\ref{lemma:splitconflation} that if $e_2(\phi,w)>0$ for some frame $\mathfrak{F}=\langle W,R^+,R^-\rangle$, $w\in\mathfrak{F}$ and $e_2$ on $\mathfrak{F}$, then $e^*_1(\phi,w)<1$. But $\phi$ does not contain $\neg$ and thus its support of falsity depends only on $e_2$ and $R^-$, whence $e^*_1$ is a $\KbiG$ valuation on $\langle W,R^-\rangle$. Thus, $\phi$ is not $\crispKbiG$-valid either.
\end{proof}

It is also clear that $\blacksquare\mathbf{1}$ and ${\sim}\blacklozenge\mathbf{0}$ are not strongly $\infoGsquare$-valid. Still, both $\blacksquare$ and $\blacklozenge$ are regular in the following sense.
\begin{proposition}\label{prop:infoGsquareregularity}
Let $\phi\rightarrow\phi'$ and $\chi\rightarrow\chi'$ be \emph{strongly valid}. Then $\blacksquare\phi\rightarrow\blacksquare\phi'$ and $\blacklozenge\chi\rightarrow\blacklozenge\chi'$ are strongly valid too.
\end{proposition}
\begin{proof}
We prove only the $\blacksquare$ case. Let $\blacksquare\phi\rightarrow\blacksquare\phi'$ be \emph{not strongly valid} in some frame $\mathfrak{F}$. Then, there is a~$w\in\mathfrak{F}$ as well as $e_1$ and $e_2$ thereon s.t.\ $e(\blacksquare\phi\rightarrow\blacksquare\phi',w)\neq(1,0)$. Since $e_1$-conditions (support of truth) of $\blacksquare$ coincide with the $\fuzzyKbiG$ semantics of $\Box$ (and since $\Box$ is obviously regular in $\fuzzyKbiG$), it suffices to check the case when $e_2(\blacksquare\phi\rightarrow\blacksquare\phi',w)>0$.

We have that
\begin{align*}
e_2(\blacksquare\phi\rightarrow\blacksquare\phi',w)>0&\text{ iff }e_2(\blacksquare\phi,w)<e_2(\blacksquare\phi',w)\\
&\text{ iff }\inf\limits_{w'\in W}\{wR^-w'\rightarrow_\mathsf{G}e_2(\phi)\}<\inf\limits_{w'\in W}\{wR^-w'\rightarrow_\mathsf{G}e_2(\phi')\}\\
&\text{ then }\exists w'\!\in\!R^-(w):e_2(\phi,w')<e_2(\phi',w')\\
&\text{ then }e_2(\phi\rightarrow\phi',w')>0
\end{align*}
The regularity of $\blacklozenge$ can be tackled similarly.
\end{proof}
\subsection{Embeddings into $\KbiG$}
In this section, we are going to construct faithful embeddings of $\KGsquare$'s and $\infoGsquare$'s into $\KbiG$. To do this in the bi-relational case, we introduce new modal operators that will enable us to convert formulas into $\neg$~NNFs.
\begin{definition}\label{def:linemodalities}
The languages $\linebimodalLtrianglesquare$ and $\lineinfobimodalLtrianglesquare$ expand $\bimodalLtrianglesquare$ and $\infobimodalLtrianglesquare$, respectively, with two new modal operators each: $\lineBox$ and $\linelozenge$ ($\linebimodalLtrianglesquare$); $\lineblacksquare$ and $\lineblacklozenge$ ($\lineinfobimodalLtrianglesquare$). Their semantics is given as follows.
\begin{center}
\begin{tabular}{rclrcl}
$e_1(\lineBox\phi,w)$&$=$&$\inf\limits_{w'\in W}\!\{wR^-w'\rightarrow_\mathsf{G}e_1(\phi,w')\}$
&
$e_2(\lineBox\phi,w)$&$=$&$\sup\limits_{w'\in W}\!\{wR^+w'\wedge_\mathsf{G}e_2(\phi,w')\}$\\
$e_1(\linelozenge\phi,w)$&$=$&$\sup\limits_{w'\in W}\!\{wR^-w'\wedge_\mathsf{G}e_1(\phi,w')\}$
&
$e_2(\linelozenge\phi,w)$&$=$&$\inf\limits_{w'\in W}\!\{wR^+w'\rightarrow_\mathsf{G}e_2(\phi,w')\}$\\
$e_1(\lineblacksquare\phi,w)$&$=$&$\inf\limits_{w'\in W}\!\{wR^-w'\!\!\rightarrow_\mathsf{G}e_1(\phi,w')\}$
&
$e_2(\lineblacksquare\phi,w)$&$=$&$\inf\limits_{w'\in W}\!\{wR^+w'\!\!\rightarrow_\mathsf{G}e_2(\phi,w')\}$\\
$e_1(\lineblacklozenge\phi,w)$&$=$&$\sup\limits_{w'\in W}\!\{wR^-w'\!\wedge_\mathsf{G}\!e_1(\phi,w')\}$
&
$e_2(\lineblacklozenge\phi,w)$&$=$&$\sup\limits_{w'\in W}\!\{wR^+w'\!\wedge_\mathsf{G}\!e_2(\phi,w')\}$
\end{tabular}
\end{center}
\end{definition}
\begin{remark}[$\neg$ NNFs in $\linebimodalLtrianglesquare$ and $\lineinfobimodalLtrianglesquare$]\label{rem:NNFs}
It is now clear that $\linebimodalLtrianglesquare$ and $\lineinfobimodalLtrianglesquare$ admit $\neg$~NNFs. Namely, the following transformations are equivalent. (For the sake of convenience, we give the transformations for all connectives and modalities, including $\coimplies$ and double G\"{o}delian negation ${\sim\sim}$.)
\begin{align}
\neg\mathbf{1}&\leftrightharpoons\mathbf{0}&\neg\mathbf{0}&\leftrightharpoons\mathbf{1}\nonumber\\
\neg\neg\phi&\leftrightharpoons\phi\nonumber&\neg{\sim}\phi&\leftrightharpoons\mathbf{1}\coimplies\neg\phi&\neg\triangle\phi&\leftrightharpoons{\sim\sim}\neg\phi&\neg{\sim\sim}\phi&\leftrightharpoons\triangle\neg\phi\\
\neg(\phi\wedge\chi)&\leftrightharpoons\neg\phi\vee\neg\chi&\neg(\phi\vee\chi)&\leftrightharpoons\neg\phi\wedge\neg\chi&\neg(\phi\rightarrow\chi)&\leftrightharpoons\neg\chi\coimplies\neg\phi&\neg(\phi\coimplies\chi)&\leftrightharpoons\neg\chi\rightarrow\neg\phi\nonumber\\
\neg\Box\phi&\leftrightharpoons\linelozenge\neg\phi&\neg\lozenge\phi&\leftrightharpoons\lineBox\neg\phi&\neg\lineBox\phi&\leftrightharpoons\lozenge\neg\phi&\neg\linelozenge\phi&\leftrightharpoons\Box\neg\phi\nonumber\\
\neg\blacksquare\phi&\leftrightharpoons\lineblacksquare\neg\phi&\neg\blacklozenge\phi&\leftrightharpoons\lineblacklozenge\neg\phi&\neg\lineblacksquare\phi&\leftrightharpoons\blacksquare\neg\phi&\neg\lineblacklozenge\phi&\leftrightharpoons\blacklozenge\neg\phi
\label{equ:lineNNFs}
\end{align}
\end{remark}

Now, since the transformation into an NNF requires the introduction of a new pair of modalities, we will need to construct our embeddings not into the mono-relational $\KbiG$ but in the \emph{bi-relational} $\KbiG$ which we will denote $\KbiG(2)$.\footnote{In what follows, we will sometimes call $\KbiG$ ‘mono-relational $\KbiG$’ and $\KbiG(2)$ ‘bi-relational $\KbiG$’. Since $\Box$ and $\lozenge$ are not interdefinable in $\KbiG$~\cite[Proposition~3]{BilkovaFrittellaKozhemiachenko2022IJCAR} even for crisp frames, and since G\"{o}del modal logics in the language with $\Box$ and~$\lozenge$ are called ‘bi-modal’, a proper moniker would be ‘tetra-modal’. Despite this, we choose ‘bi-relational’ to designate that each pair of modalities correspond to one relation of the two.} The language of $\KbiG(2)$, $\bimodalLtriangle(2)$ contains two pairs of modalities: $\Box_1$, $\lozenge_1$, $\Box_2$, and $\lozenge_2$. It is also clear that the axiomatisation of both $\crispKbiG(2)$ and $\fuzzyKbiG(2)$ can be obtained from $\HcrispKbiG$ and $\HfuzzyKbiG$ by replicating the modal axioms and rules for $\Box_2$ and $\lozenge_2$. To construct these embeddings, we are going to reduce every $\linebimodalLtrianglesquare$ or $\lineinfobimodalLtrianglesquare$ formula $\phi$ not to one but to \emph{two} $\bimodalLtriangle(2)$ formulas: $\phi^*$ and $\phi^\partial$. Moreover, we will treat $\coimplies$ as a~basic connective to make the size of embeddings linear. Since bi-relational fuzzy frames constitute the largest class of frames, these embeddings will suffice for all paraconsistent modal logics in~Fig.~\ref{fig:logics}.
\begin{definition}\label{def:translations}
Let $\phi\in\linebimodalLtrianglesquare\cup\lineinfobimodalLtrianglesquare$ be in $\neg$~NNF. Then $\phi^*$ is the result of replacing every literal $\neg p$ occurring in $\phi$ with a new variable $p^*$. Given $\chi^*\in\linebimodalLtrianglesquare\cup\lineinfobimodalLtrianglesquare$, we define $\chi^\partial$ as follows.
\begin{align*}
p^\partial&=p^*&(p^*)^\partial&=p\\
\mathbf{1}^\partial&=\mathbf{0}&\mathbf{0}^\partial&=\mathbf{1}\\
({\sim}\chi)^\partial&=\mathbf{1}\coimplies\chi^\partial&(\triangle\chi)^\partial&={\sim\sim}\chi^\partial\\
(\chi_1\wedge\chi_2)^\partial&=\chi^\partial_1\vee\chi^\partial_2&(\chi_1\vee\chi_2)^\partial&=\chi^\partial_1\wedge\chi^\partial_2\\
(\chi_1\rightarrow\chi_2)^\partial&=\chi^\partial_2\coimplies\chi^\partial_1&(\chi_1\coimplies\chi_2)^\partial&=\chi^\partial_2\rightarrow\chi^\partial_1\\
(\Box\chi)^\partial&=\lozenge_2\chi^\partial&(\blacksquare\chi)^\partial&=\Box_2\chi^\partial\\
(\lozenge\chi)^\partial&=\Box_2\chi^\partial&(\blacklozenge\psi)^\partial&=\lozenge_2\chi^\partial\\
(\lineBox\chi)^\partial&=\lozenge_1\chi^\partial&(\lineblacksquare\chi)^\partial&=\Box_1\chi^\partial\\
(\linelozenge\chi)^\partial&=\Box_1\chi^\partial&(\lineblacklozenge\chi)^\partial&=\lozenge_1\chi^\partial
\end{align*}
\end{definition}
\begin{convention}\label{conv:monorelationallanguage}
In what follows, if we deal with \emph{mono-relational frames}, we will assume that $^\partial$ translation does not add modalities of the form $\overline{\heartsuit}$. Indeed, it is clear that $\heartsuit p\leftrightarrow\overline{\heartsuit} p$ is strongly valid on any mono-relational frame for every pair of modalities $\heartsuit$ and $\overline{\heartsuit}$.
\end{convention}
\begin{convention}\label{conv:circtranslation}~
\begin{enumerate}
\item Let $\phi\in\lineinfobimodalLtrianglesquare$ be $\neg$-free. We use $\phi^\circ$ to denote the formula obtained by replacing all $\blacksquare$'s, $\blacklozenge$'s, $\lineblacksquare$'s and $\lineblacklozenge$'s in $\phi$ with $\Box$'s, $\lozenge$'s, $\lineBox$'s, and $\linelozenge$'s, respectively.
\item Let $\phi\in\linebimodalLtrianglesquare$. We use $\phi^{+\bullet}$ to denote the formula obtained by replacing all $\Box$'s, $\lozenge$'s, $\lineBox$'s, and $\linelozenge$'s in $\phi$ with $\blacksquare$'s, $\blacklozenge$'s, $\lineblacksquare$'s and $\lineblacklozenge$'s, respectively.
\end{enumerate}
\end{convention}

It is easy to see that $^*$-transformation preserves validity on frames.
\begin{lemma}\label{lemma:negNNF}
Let $\phi\in\linebimodalLtrianglesquare\cup\lineinfobimodalLtrianglesquare$ be in $\neg$~NNF. Then for every pointed frame $\langle\mathfrak{F},w\rangle$, it holds that
\begin{align*}
\mathfrak{F},w\models^+_{\KGsquare}\phi&\text{ iff }\mathfrak{F},w\models^+_{\KGsquare}\phi^*&\mathfrak{F},w\models^-_{\KGsquare}\phi&\text{ iff }\mathfrak{F},w\models^-_{\KGsquare}\phi^*\\
\mathfrak{F},w\models^+_{\infoGsquare}\phi&\text{ iff }\mathfrak{F},w\models^+_{\infoGsquare}\phi^*&\mathfrak{F},w\models^-_{\infoGsquare}\phi&\text{ iff }\mathfrak{F},w\models^-_{\infoGsquare}\phi^*
\end{align*}
\end{lemma}
\begin{proof}
Let $\phi$ be in $\neg$~NNF and $e_1$ an arbitrary valuation. We construct $e^*_1$ as follows: $e^*_1(p,u)=e_1(p,u)$ and $e^*_1(p^*,u)=e_2(p,u)=e_1(\neg p,u)$ for all $u\in\mathfrak{F}$. It is clear that $e_1(\phi,w)=e^*_1(\phi^*,w)$. For the converse, we define $e'_1(p,u)=e_1(p,u)$ and $e'_2(p,u)=e'_1(\neg p,u)=e_1(p^*,u)$ for all $u\in\mathfrak{F}$. Again, it is easy to check by induction that $e_1(\phi^*,w)=e'_1(\phi,w)$.

The $e_2$-validity can be considered in the same way.
\end{proof}

The next two lemmas establish that $^\partial$ preserves values.
\begin{lemma}\label{lemma:partialembedding}
Let $\mathfrak{M}=\langle W,R^+,R^-,e_1,e_2\rangle$ be $\fuzzybirelKGsquare$ model and let $\mathfrak{M}^\partial\!=\!\langle W,R_1,R_2,e^\partial\rangle$ be a $\fuzzyKbiG(2)$ model s.t.\ $R_1=R^+$, $R_2=R^-$, $e^\partial(p,w)=e_2(p^*,w)$, and $e^\partial(p^*,w)=e_2(p,w)$. Then, $e_2(\phi^*)=e^\partial(\phi^\partial)$ for every $\phi^*\in\linebimodalLtrianglesquare$.
\end{lemma}
\begin{proof}
We proceed by induction on $\phi^*$. The basis cases of $p$ and $p^*$ variables hold by the construction of $\mathfrak{M}^\partial$. The cases of propositional connectives hold by a simple application of the induction hypothesis. Consider, for example, $\phi^*=\chi_1\wedge\chi_2$.
\begin{align*}
e_2(\chi_1\wedge\chi_2,w)&=e_2(\chi_1,w)\vee_\mathsf{G}e_2(\chi_2,w)\\
&=e^\partial(\chi^\partial_1,w)\vee_\mathsf{G}e^\partial(\chi^\partial_2,w)\tag{by IH}\\
&=e^\partial(\chi^\partial_1\vee\chi^\partial_2,w)\\
&=e^\partial((\chi_1\wedge\chi_2)^\partial,w)\tag{by Definition~\ref{def:translations}}
\end{align*}

The cases of modal formulas can also be tackled in a similar manner, whence we consider only $\phi^*=\linelozenge\chi$.
\begin{align*}
e_2(\linelozenge\chi,w)&=\inf\limits_{w'\in W}\{wR^+w'\rightarrow_\mathsf{G}e_2(\chi,w)\}\\
&=\inf\limits_{w'\in W}\{wR^+w'\rightarrow_\mathsf{G}e^\partial(\chi^\partial,w)\}\tag{by IH}\\
&=e^\partial(\Box\chi^\partial,w)\\
&=e^\partial((\linelozenge\chi)^\partial,w)\tag{by Definition~\ref{def:translations}}
\end{align*}
\end{proof}
\begin{lemma}\label{lemma:proptoembedding}
Let $\mathfrak{M}=\langle W,R^+,R^-,e_1,e_2\rangle$ be a~$\fuzzyinfoGsquare$ model and let $\mathfrak{M}^\partial=\langle W,R_1,R_2,e^\partial\rangle$ be a $\fuzzyKbiG(2)$ model s.t.\ $R_1=R^+$, $R_2=R^-$, $e^\partial(p,w)=e_2(p^*,w)$, and $e^\partial(p^*,w)=e_2(p,w)$. Then, $e_2(\phi^*)=e^\partial(\phi^\partial)$ for every $\phi^*\in\lineinfobimodalLtrianglesquare$.
\end{lemma}
\begin{proof}
Analogously to Lemma~\ref{lemma:partialembedding}.
\end{proof}

We can also obtain a result similar to Lemma~\ref{lemma:negNNF} but regarding $^\partial$.
\begin{proposition}\label{prop:partialproptoequivalence}
Let $\heartsuit\in\{\Box,\lozenge\}$ and $\heartsuit_1$ and $\heartsuit_2$ be evaluated the same as $\heartsuit$ and $\overline{\heartsuit}$, respectively. Then the following statements hold.
\begin{enumerate}
\item Assume that $\phi\in\linebimodalLtrianglesquare$. Then $\mathfrak{F},w\models_{\KGsquare}\phi$ iff $\mathfrak{F},w\models_{\KGsquare}{\sim}\phi^\partial$.
\item Assume that $\phi\in\lineinfobimodalLtrianglesquare$. Then $\mathfrak{F},w\models_{\infoGsquare}\phi$ iff $\mathfrak{F},w\models_{\infoGsquare}\!({\sim}\phi^\partial)^{+\bullet}$.
\end{enumerate}
\end{proposition}
\begin{proof}
We prove only 1.\ since 2.\ can be shown in the same way. Let $\mathfrak{F},w\not\models_{\KGsquare}\phi$. Then either (1) there is $e_1$ on $\mathfrak{F}$ s.t.\ $e_1(\phi,w)<1$ or (2) there is $e_2$ on $\mathfrak{F}$ s.t.\ $e_2(\phi,w)>0$. In the second case, the statement holds immediately by Lemmas~\ref{lemma:negNNF} and~\ref{lemma:partialembedding} since $\KbiG$-valuations are defined in the same way as $e_1$-valuations in $\KGsquare$. Namely, in this case, we will have that there is a valuation $e'_1$~on $\mathfrak{F}$ s.t.\ $e^\partial_1(\phi^\partial,w)>0$, whence, $e^\partial_1({\sim}\phi^\partial,w)=0$.

In the first case, we proceed as follows. First, transform $\phi$ into $\phi^*$ (we can do this because of Lemma~\ref{lemma:negNNF}). For a valuation $e_1$ on $\mathfrak{F}$, we define $e^\partial_2$ as expected: $e^\partial_2(p,w)=e_1(p,w)$ and $e^\partial_2(p^*,w)=e_1(p,w)$. Now, we prove by induction that $e_1(\phi^*,w)=e^\partial_2((\phi^*)^\partial,w)$.

This holds for the propositional variables by construction. The cases of propositional connectives are also simple (we consider $\phi^*={\sim}\chi$).
\begin{align*}
e_1({\sim}\chi,w)&=e_1(\chi,w)\rightarrow_\mathsf{G}0\\
&=e^\partial_2(\chi^\partial,w)\rightarrow_\mathsf{G}0\tag{by IH}\\
&=e^\partial_2(\mathbf{1}\coimplies\chi^\partial,w)\\
&=e^\partial_2(({\sim}\chi)^\partial,w)\tag{by Definition~\ref{def:translations}}
\end{align*}
Since the cases of modalities are similar, we consider only $\phi^*=\lozenge\chi$.
\begin{align*}
e_1(\lozenge\chi,w)&=\sup\limits_{w'\in W}\{wR^+w'\wedge_\mathsf{G}e_1(\chi,w)\}\\
&=\sup\limits_{w'\in W}\{wR^+w'\wedge_\mathsf{G}e_2(\chi^\partial,w)\}\tag{by IH}\\
&=e_2(\lineBox\chi^\partial,w)\\
&=e_2((\lozenge\chi)^\partial,w)\tag{by Definition~\ref{def:translations}}
\end{align*}

For the converse, let $\mathfrak{F},w\not\models_{\KGsquare}{\sim}\phi^\partial$. Then we have that $\mathfrak{F},w\not\models{\sim}({\sim}\phi^\partial)^\partial$ by the previous part. I.e., $\mathfrak{F},w\not\models_{\KGsquare}{\sim}(\mathbf{1}\coimplies\phi^\partial)^\partial$, whence, $\mathfrak{F},w\not\models_{\KGsquare}{\sim}((\phi^\partial)^\partial\rightarrow\mathbf{0})$. Now observe from Definition~\ref{def:translations} that $(\psi^\partial)^\partial=\psi$ for every $\psi$. Hence, $\mathfrak{F},w\not\models_{\KGsquare}{\sim\sim}\phi^*$, from which we have $\mathfrak{F},w\not\models_{\KGsquare}\phi^*$. Now, by Lemma~\ref{lemma:negNNF}, we obtain $\mathfrak{F},w\not\models_{\KGsquare}\phi$, as required.

Part 2.\ can be proved similarly but we would need to use Lemmas~\ref{lemma:negNNF} and~\ref{lemma:proptoembedding}.
\end{proof}

We can now establish the embedding results.
\begin{theorem}\label{theorem:validityembedding}~
\begin{enumerate}
\item Let $\mathfrak{F}\!=\!\langle W,R^+,R^-\rangle$ be a bi-relational frame and $w\in\mathfrak{F}$. Then for all $\phi\in\linebimodalLtrianglesquare$, it holds that
\begin{align*}
\mathfrak{F},w\models_{\fuzzybirelKGsquare}\!\phi&\text{ iff }\mathfrak{F},w\models_{\fuzzyKbiG(2)}\!\phi^*\!\wedge\!{\sim}\phi^\partial
\end{align*}
where $\heartsuit_1$'s are associated to $R^+$ and $\heartsuit_2$'s to $R^-$ ($\heartsuit\in\{\Box,\lozenge\}$).
\item Let $\mathfrak{F}\!=\!\langle W,R^+,R^-\rangle$ be a bi-relational frame and $w\in\mathfrak{F}$. Then for all $\phi\in\lineinfobimodalLtrianglesquare$, it holds that
\begin{align*}
\mathfrak{F},w\models_{\fuzzyinfoGsquare}\!\phi&\text{ iff }\mathfrak{F},w\models_{\fuzzyKbiG(2)}\!(\phi^*)^\circ\!\wedge\!{\sim}\phi^\partial
\end{align*}
where $\heartsuit_1$'s are associated to $R^+$ and $\heartsuit_2$'s to $R^-$ ($\heartsuit\in\{\Box,\lozenge\}$).
\end{enumerate}
\end{theorem}
\begin{proof}
We begin with 1. One can see that the transformations in Remark~\ref{rem:NNFs} are equivalent. Furthermore, from Lemma~\ref{lemma:negNNF}, we have that $\phi^*$ is $\KGsquare$-valid on a given frame $\mathfrak{F}$ iff $\phi$ is $\KGsquare$-valid on~$\mathfrak{F}$.

Now, it is clear that $\phi$ is \emph{$e_1$-valid} on $\langle\mathfrak{F},w\rangle$ iff $\phi^*$ is $\KbiG$-valid on $\langle\mathfrak{F},w\rangle$. Since $\phi^*$ is \emph{strongly valid} on $\langle\mathfrak{F},w\rangle$, it is $e_1$-valid on $\langle\mathfrak{F},w\rangle$ as well. Likewise, using Lemma~\ref{lemma:partialembedding}, we obtain that $\phi$ is $e_2$-valid on $\langle\mathfrak{F},w\rangle$ iff ${\sim}\phi^\partial$ is $\KbiG$-valid on $\langle\mathfrak{F},w\rangle$. Namely, $\phi$ is $e_2$-valid iff $\phi^*$ is $e_2$-valid (Lemma~\ref{lemma:negNNF}). Assume that $\phi^*$ is \emph{not} $e_2$-valid on $\langle\mathfrak{F},w\rangle$. By Lemma~\ref{lemma:partialembedding}, it is immediate that $e^\partial(\phi^\partial)>0$, whence, ${\sim}\phi^\partial$ is not $\KbiG$-valid on $\langle\mathfrak{F},w\rangle$. For the converse, let $e'({\sim}\phi,w)<1$ (whence, $e'(\phi,w)>0$) for some $\KbiG$ valuation $e'$ on $\mathfrak{F}$. Pick a $\KGsquare$ valuation $e'_2$ s.t.\ $e'^\partial=e'_2$. We have that $e'_2(\phi^*,w)>0$ by Lemma~\ref{lemma:partialembedding}. Hence, $\phi$ is not $e_2$-valid by Lemma~\ref{lemma:negNNF}.

Since $\phi$ is strongly $\KGsquare$-valid on $\langle\mathfrak{F},w\rangle$ iff it is both $e_1$- and $e_2$-valid thereon, the result follows.

The proof of 2.\ is the same as that of 1.\ but instead of Lemma~\ref{lemma:partialembedding}, we use Lemma~\ref{lemma:proptoembedding}.
\end{proof}
\begin{remark}
Note from Theorem~\ref{theorem:validityembedding} that if $\phi$ is $\neg$-free (i.e., $\phi^*=\phi$), we have that $\mathfrak{F}\models_{\KGsquare}\phi$ iff $\mathfrak{F}\models_{\KbiG}\phi\wedge{\sim}\phi^\partial$ and $\mathfrak{F}\models_{\infoGsquare}\phi$ iff $\mathfrak{F}\models_{\KbiG}\phi\wedge{\sim}\phi^\partial$. This means that $\phi$ defines the same class of frames $\mathbb{K}$ in $\KbiG$ and $\KGsquare$ ($\infoGsquare$, respectively) iff $\mathbb{K}\models_{\KbiG}\triangle\phi\leftrightarrow{\sim}\phi^\partial$ ($\mathbb{K}\models_{\KbiG}\triangle\phi\leftrightarrow{\sim}\phi^\partial$).
\end{remark}

We can also establish the embeddings of $\KGsquare$ and $\infoGsquare$ entailments into $\KbiG$ entailment.
\begin{theorem}\label{theorem:entailmentsembedding}~
\begin{enumerate}
\item Let $\Gamma\cup\{\phi\}\subseteq\linebimodalLtrianglesquare$. Then $\Gamma\models_{\KGsquare}\phi$ iff $\{\psi^*:\psi\in\Gamma\}\models_{\KbiG(2)}\phi^*$ and there exists a finite set $\Gamma^\partial\subseteq\{\psi^\partial:\psi\in\Gamma\}$ s.t.\ $\phi^\partial\models_{\KbiG(2)}\bigvee\limits_{\psi^\partial\in\Gamma^\partial}\psi^\partial$.
\item Let $\Gamma\cup\{\phi\}\subseteq\lineinfobimodalLtrianglesquare$. Then $\Gamma\models_{\infoGsquare}\phi$ iff $\{\psi^*:\psi\in\Gamma\}\models_{\KbiG(2)}\phi^*$ and there exists a finite set $\Gamma^\partial\subseteq\{\psi^\partial:\psi\in\Gamma\}$ s.t.\ $\phi^\partial\models_{\KbiG(2)}\bigvee\limits_{\psi^\partial\in\Gamma^\partial}\psi^\partial$.
\end{enumerate}
\end{theorem}
\begin{proof}
Consider 1., let $\Gamma\cup\{\phi\}\subseteq\bimodalLtrianglesquare$, and assume that $\Gamma\models_{\fuzzybirelKGsquare}\phi$. By Definition~\ref{def:paraconsistentsemantics}, it means that $\inf\{e_1(\psi,w):\psi\in\Gamma\}\leq e_1(\phi,w)$ and $\sup\{e_2(\psi,w):\psi\in\Gamma\}\geq e_2(\phi,w)$ for every $\fuzzybirelKGsquare$ model $\mathfrak{M}=\langle W,R^+,R^-,e_1,e_2\rangle$ and every $w\in W$. Now, by Lemma~\ref{lemma:negNNF}, we have that $\inf\{e(\psi^*,w):\psi\in\Gamma\}\leq e(\phi^*,w)$ for every $\fuzzyKbiG(2)$ model $\mathfrak{M}=\langle W,R^+,R^-,e\rangle$ and every $w\in W$, i.e., $\{\psi^*:\psi\in\Gamma\}\models_{\KbiG(2)}\phi^*$. Furthermore, by Lemma~\ref{lemma:partialembedding}, we have that $\sup\{e(\psi^\partial,w):\psi\in\Gamma\}\geq e(\phi^\partial,w)$ for every $\fuzzyKbiG(2)$ model $\mathfrak{M}=\langle W,R^+,R^-,e\rangle$ and every $w\in W$. Note, however, that since $\fuzzyKbiG$ is strongly complete (Theorem~\ref{theorem:KbiGstrongcompleteness}), it is also compact. Thus, there must exist some finite $\Gamma^\partial\subseteq\{\psi^\partial:\psi\in\Gamma\}$ s.t.\ $\phi^\partial\models_{\KbiG(2)}\bigvee\limits_{\psi^\partial\in\Gamma^\partial}\psi^\partial$, as required.

For the converse, let $\Gamma\not\models_{\fuzzybirelKGsquare}\phi$. Then, there are some $\fuzzybirelKGsquare$ model $\mathfrak{M}=\langle W,R^+,R^-,e_1,e_2\rangle$ and $w\in W$ s.t.\ (1) $\inf\{e_1(\psi,w):\psi\in\Gamma\}>e_1(\phi,w)$ or (2) $\sup\{e_2(\psi,w):\psi\in\Gamma\}<e_2(\phi,w)$. It is clear from Lemmas~\ref{lemma:negNNF} and~\ref{lemma:partialembedding} that $\{\psi^*:\psi\in\Gamma\}\not\models_{\KbiG(2)}\phi^*$ in the first case, and that for any finite $\Gamma^\partial\subseteq\{\psi^\partial:\psi\in\Gamma\}$, we have $\phi^\partial\not\models_{\KbiG(2)}\bigvee\limits_{\psi^\partial\in\Gamma^\partial}\psi^\partial$ in the second case.

Part 2.\ can be dealt with similarly, using Lemma~\ref{lemma:proptoembedding} instead of Lemma~\ref{lemma:partialembedding}.
\end{proof}
\subsection{Frame definability}
In this section, we show further results concerning the definability of frames in $\birelKGsquare$. We begin with the characterisation of $\KGsquare$- and $\infoGsquare$-definable classes of frames.
\begin{definition}\label{def:framedefinability}
Let $\mathbb{S}$ be a class of frames, $\mathbf{L}$ be a logic and $\mathcal{L}_\mathbf{L}$ its language. We say that $\phi\in\mathcal{L}_\mathbf{L}$ \emph{defines $\mathbb{S}$ in $\mathbf{L}$} iff for every frame $\mathfrak{F}$, it holds that
\begin{align*}
\mathfrak{F}\in\mathbb{S}&\text{ iff }\mathfrak{F}\models_\mathbf{L}\phi
\end{align*}
A class of frames is \emph{definable in $\mathbf{L}$} iff there is an~$\mathcal{L}_\mathbf{L}$ formula that defines it.
\end{definition}
\begin{corollary}\label{cor:definabilitycharacterisation}~
\begin{enumerate}
\item All $\KGsquare$- and $\infoGsquare$-definable classes of frames are $\KbiG$-definable.
\item A class of mono- or bi-relational frames $\mathbb{S}$ is $\KGsquare$-definable iff there is $\phi\in\bimodalLtriangle(2)$ s.t.\ $\phi\wedge{\sim}\phi^\partial$ defines $\mathbb{S}$ in $\KbiG$.
\item A class of mono- or bi-relational frames $\mathbb{S}$ is $\infoGsquare$-definable iff there is $\phi\in\bimodalLtriangle(2)$ s.t.\ $\phi\wedge{\sim}\phi^\partial$ defines $\mathbb{S}$ in $\KbiG$.
\end{enumerate}
\end{corollary}
\begin{proof}
1.\ follows immediately from Theorems~\ref{theorem:validityembedding} and~\ref{theorem:entailmentsembedding}.

Let us now prove 2.\ as 3.\ can be shown similarly. Assume that $\chi$ defines $\mathbb{S}$ in $\KGsquare$. By Theorem~\ref{theorem:validityembedding}, it follows that $\mathbb{S}\models_{\KbiG}\chi^*\wedge{\sim}\chi^\partial$ and $\mathfrak{F}\not\models_{\KbiG}\chi^*\wedge{\sim}\chi^\partial$ for every $\mathfrak{F}\notin\mathbb{S}$. For the converse, assume that $\phi\wedge{\sim}\phi^\partial$ defines $\mathbb{S}$ in $\KbiG$. Since $\phi$ does not contain $\neg$, $\phi^*=\phi$. Thus, again, by Theorem~\ref{theorem:validityembedding}, we have that $\phi$ defines $\mathbb{S}$ in $\KGsquare$.
\end{proof}

In the remainder of the section, we are going to be concerned with $\KGsquare$ since informational modalities are non-standard. First of all, we can see that \emph{mono-relational} frames are $\KGsquare$-definable\footnote{The $\infoGsquare$-definability of mono-relational frames can be found in~\cite{BilkovaFrittellaKozhemiachenko2023nonstandard}.} in the expected manner.
\begin{proposition}\label{prop:birelKGsquare1relationdefinable}
Let $\mathfrak{F}=\langle W,R^+,R^-\rangle$. Then $\mathfrak{F}\models\Box p\leftrightarrow\neg\lozenge\neg p$ iff $R^+=R^-$.
\end{proposition}
\begin{proof}
Let $R^+=R^-=R$, we have
\begin{align*}
e_1(\neg\lozenge\neg p,w)&=e_2(\lozenge\neg p,w)\\
&=\inf\limits_{w'\in W}\{wRw'\rightarrow_\mathsf{G}e_2(\neg p,w')\}\\
&=\inf\limits_{w'\in W}\{wRw'\rightarrow_\mathsf{G}e_1(p,w')\}\\
&=e_1(\Box p,w)
\end{align*}
$e_2$ can be tackled similarly.

Now let $R^+\neq R^-$, i.e., $wR^+w'=x$ and $wR^-w'=y$ for some $w,w'\in\mathfrak{F}$, and assume w.l.o.g.\ that $x>y$. We define the values of $p$ as follows: $e(p,w'')=(1,0)$ for all $w''\neq w'$ and $e(p,w')=(x,y)$. It is clear that $e(\Box p,w)=(1,0)$ but $e(\neg\lozenge\neg p,w)=(y,x)\neq(1,0)$, as required.
\end{proof}
\begin{definition}[$\pm$-counterparts of frames]\label{def:framecounterparts}
Let $\mathbb{K}$ be a class of crisp \emph{mono-relational} frames.
\begin{enumerate}
\item A class $\mathbb{K}^+$ of \emph{bi-relational} crisp frames is the \emph{$+$-counterpart} (‘plus-counterpart’) of $\mathbb{K}$ iff for every bi-relational frame $\langle U,S,R^-\rangle$, it holds that
\begin{align*}
\langle U,S,R^-\rangle\in\mathbb{K}^+&\text{ iff }\langle U,S\rangle\in\mathbb{K}
\end{align*}
\item A class $\mathbb{K}^-$ of \emph{bi-relational} crisp frames is the \emph{$-$-counterpart} (‘minus-counterpart’) of $\mathbb{K}$ iff for every bi-relational frame $\langle U,R^+,S\rangle$, it holds that
\begin{align*}
\langle U,R^+,S\rangle\in\mathbb{K}^-&\text{ iff }\langle U,S\rangle\in\mathbb{K}
\end{align*}
\item A class $\mathbb{K}^\pm$ of \emph{bi-relational} crisp frames is the \emph{$\pm$-counterpart} (‘plus-minus-counterpart’) of $\mathbb{K}$ iff for every bi-relational frame $\langle W,R^+,R^-\rangle$, it holds that
\begin{align*}
\langle W,R^+,R^-\rangle\in\mathbb{K}^\pm&\text{ iff }\left[\begin{matrix}\langle W,R^+\rangle\in\mathbb{K}\\\text{ and }\\\langle W,R^-\rangle\in\mathbb{K}\end{matrix}\right]
\end{align*}
\end{enumerate}
\end{definition}
\begin{corollary}\label{cor:definabilitypreservation}
Let $\phi$ be a $\neg$-free formula that defines a class of frames $\mathfrak{F}=\langle W,R\rangle$, $\mathbb{K}$ in $\crispKbiG$ and let $\mathbb{K}^\pm$ be the $\pm$-counterpart of $\mathbb{K}$. Then $\phi$ defines $\mathbb{K}^\pm$ in $\crispbirelKGsquare$.
\end{corollary}
\begin{proof}
Assume that $\phi$ \emph{does not} define $\mathbb{K}^\pm$ in $\crispbirelKGsquare$. Then, either (1) there is $\mathfrak{F}\notin\mathbb{K}^\pm$ s.t.\ $\mathfrak{F}\models_{\KGsquare}\phi$ or (2) $\mathfrak{H}\not\models_{\KGsquare}\phi$ for some $\mathfrak{H}\in\mathbb{K}^\pm$. Since $\phi$ defines $\mathbb{K}$ in $\KbiG$, it is clear that $\mathfrak{F},\mathfrak{H}\in\mathbb{K}^+$. Thus, we need to reason for contradiction in the case when $\mathfrak{F}\notin\mathbb{K}^-$ or $\mathfrak{H}\notin\mathbb{K}^-$. We prove only (1) as (2) can be tackled in a dual manner.

Observe that $e_2(\phi,w)\!=\!0$ for every $w\!\in\!\mathfrak{F}$ and $e_2$ on $\mathfrak{F}\!=\!\langle W,R^+,R^-\rangle$. But then, by Lemma~\ref{lemma:splitconflation}, we have that $e^*_1(\phi,w)=1$ for every $w\in\mathfrak{F}^*$ and $e^*_1$ on $\mathfrak{F}^*=\langle W,R^-,R^+\rangle$\footnote{Note that $R^+$ and $R^-$ are now swapped.}. Thus, since for every $e^*_1$, there is $e_2$ from which it could be obtained, $\phi$ is $\KbiG$-valid on a frame $\langle W,R^-\rangle\notin\mathbb{K}$. Hence, $\phi$~does not define $\mathbb{K}$ in $\KbiG$ either. The result follows.
\end{proof}
\begin{corollary}\label{cor:partialnondefinability}
Let $\mathbb{K}$ be a class of \emph{crisp mono-relational frames} $\langle W,R\rangle$. Let further $\mathbb{K}(2)$ be a class of crisp bi-relational frames s.t.\ exactly one of the following holds:
\begin{enumerate}
\item $\mathbb{K}(2)=\{\langle W,R^+,R^-\rangle:\langle W,R^+\rangle\in\mathbb{K}\text{ and there is some }\langle W,R^-\rangle\notin\mathbb{K}\}$, or
\item $\mathbb{K}(2)=\{\langle W,R^+,R^-\rangle:\langle W,R^-\rangle\in\mathbb{K}\text{ and there is some }\langle W,R^+\rangle\notin\mathbb{K}\}$.
\end{enumerate}
Then, $\mathbb{K}(2)$ is not $\KGsquare$-definable.
\end{corollary}
\begin{proof}
Since both cases can be shown similarly, we prove only 1.\ and reason for the contradiction. Assume that $\phi$ defines $\mathbb{K}(2)$, and let $\mathfrak{F}\in\mathbb{K}(2)$ with $\mathfrak{F}=\langle W,R^+,R^-\rangle$ be s.t.\ $\langle W,R^-\rangle\notin\mathbb{K}$. Now denote $\mathfrak{F}^*=\langle W,R^-,R^+\rangle$. Clearly, $\mathfrak{F}^*\notin\mathbb{K}(2)$. However, by Lemma~\ref{lemma:splitconflation}, we have that $\mathfrak{F}^*\models\phi$, i.e., $\phi$ does not define $\mathbb{K}$. A~con\-t\-ra\-dic\-tion.
\end{proof}

Note, however, that the above statement fails for \emph{fuzzy frames}. Indeed, it is possible to define a~class of frames where only one relation is crisp.
\begin{proposition}\label{prop:crispdefinition}
Let $\mathfrak{F}=\langle W,R^+,R^-\rangle$.
\begin{enumerate}
\item $R^+$ is \emph{crisp} iff $\mathfrak{F}\models\triangle\Box p\rightarrow\Box\triangle p$.
\item $R^-$ is \emph{crisp} iff $\mathfrak{F}\models\lozenge{\sim\sim}p\rightarrow{\sim\sim}\lozenge p$.
\end{enumerate}
\end{proposition}
\begin{proof}
Observe, first of all, that $e_i(\triangle\phi,w),e_i({\sim\sim}\phi,w)\in\{0,1\}$ for every $\phi$ and $i\in\{1,2\}$. Now let $R^+$ be crisp. We have
\begin{align*}
e_1(\triangle\Box p,w)=1&\text{ then }e_1(\Box p,w)=1\\
&\text{ then }\inf\{e_1(p,w'):wR^+w'\}=1\tag{$R^+$ is crisp}\\
&\text{ then }\inf\{e_1(\triangle p,w'):wR^+w'\}=1\\
&\text{ then }e_1(\Box\triangle p,w)=1
\end{align*}
\begin{align*}
e_2(\triangle\Box p,w)=0&\text{ then }e_2(\Box p,w)=0\\
&\text{ then }\sup\limits_{w'\in W}\{wR^-w'\wedge_\mathsf{G}e_2(p,w')\}=0\\
&\text{ then }\sup\limits_{w'\in W}\{wR^-w'\wedge_\mathsf{G}e_2(\triangle p,w')\}=0\\
&\text{ then }e_2(\Box\triangle p,w)=0
\end{align*}
For the converse, let $wR^+w'=x$ with $0<x<1$. We set $e(p,w')=(x,0)$ and $e(p,w'')=(1,0)$ elsewhere. It is clear that $e_1(\triangle\Box p,w)=1$ but $e_1(\Box\triangle p,w)=0$. Thus, $e_1(\triangle\Box p\rightarrow\Box\triangle p,w)\neq1$, as required.

The case of $R^-$ is considered dually. For crisp $R^-$, we have
\begin{align*}
e_1(\lozenge{\sim\sim}p,w)=1&\text{ then }\sup\limits_{w'\in W}\{wR^-w'\wedge_\mathsf{G}e_1({\sim\sim}p,w')\}=1\\
&\text{ then }\sup\limits_{w'\in W}\{wR^-w'\wedge_\mathsf{G}e_1(p,w')\}>0\\
&\text{ then }e_1(\lozenge p,w)>0\\
&\text{ then }e_1({\sim\sim}\lozenge p,w)=1
\end{align*}
\begin{align*}
e_2(\lozenge{\sim\sim}p,w)=0&\text{ then }\inf\{e_2({\sim\sim}p,w'):wR^-w'\}=0\tag{$R^-$ is crisp}\\
&\text{ then }\inf\{e_2(p,w'):wR^-w'\}<1\\
&\text{ then }e_2(\lozenge p,w)<1\\
&\text{ then }e_2({\sim\sim}\lozenge p,w)=0
\end{align*}
For the converse, let $wR^-w'=y$ with $y\in(0,1)$. We set $e(p,w')=(1,y)$ and $e(p,w'')=(1,0)$ elsewhere. It is clear that $e_2(\lozenge{\sim\sim} p,w)=0$ but $e_2({\sim\sim}\lozenge p,w)=1$. Thus, $e_2(\lozenge{\sim\sim}p\rightarrow{\sim\sim}\lozenge p,w)\neq0$, as required.
\end{proof}
\begin{remark}\label{rem:nocontradiction}
Note that there is no contradiction between Propositions~\ref{prop:partialproptoequivalence} and~\ref{prop:crispdefinition}. Indeed, $(\triangle\Box p\rightarrow\Box\triangle p)^\partial=\linelozenge{\sim\sim}p\rightarrow{\sim\sim}\linelozenge p$ in the bi-relational case. If, on the other hand, the underlying frame is \emph{mono-relational}, it either validates or refutes both formulas in Proposition~\ref{prop:crispdefinition}.
\end{remark}
\section{Transfer\label{sec:transfer}}
In~\cite{BilkovaFrittellaKozhemiachenko2023IGPL}, we studied \emph{transferrable formulas}, i.e., formulas over $\{\mathbf{0},\wedge,\vee,\rightarrow,\Box,\lozenge\}$ that are \emph{classically} valid on some crisp frame $\mathfrak{F}$ iff they are $\KbiG$-valid on $\mathfrak{F}$. Classical modal logic does not support fuzzy frames, however. Furthermore, $\fuzzyKGsquare$ does not extend $\fuzzyKbiG$ (recall Section~\ref{ssec:paraconsistentsemantics} and Fig.~\ref{fig:KGbirelnoextension}). Thus, it makes sense to study the transfer from $\fuzzyKbiG$ to $\fuzzyKGsquare$. In addition, since modalities in $\linebimodalLtrianglesquare$ \emph{do not treat $R^+$ and $R^-$ independently}, it makes sense to ask which formulas can be transferred from the \emph{mono-relational} $\fuzzyKbiG$ to $\fuzzybirelKGsquare$.

Let us now introduce the notions of transfer.
\begin{definition}[Transfer]\label{def:transfer}
We call $\phi\in\bimodalLtriangle(2)$ \emph{transferrable} iff the following condition holds for every pointed mono- or bi-relational frame $\langle\mathfrak{F},w\rangle$.
\begin{align*}
\mathfrak{F},w\models_{\KbiG}\phi&\text{ iff }\mathfrak{F},w\models_{\KGsquare}\phi
\end{align*}
\end{definition}
\begin{definition}[Bi-transfer]\label{def:bitransfer}
Let $\mathfrak{F}=\langle W,R,w\rangle$ and $\mathfrak{F}'=\langle W,S,w\rangle$ be two pointed frames. Let $\mathfrak{F}^{R,S}=\langle W,R^+_{R,S},R^-_{R,S}\rangle$ and $\mathfrak{F}^{S,R}=\langle W,R^+_{S,R},R^-_{S,R}\rangle$ be s.t.\
\begin{align*}
\forall w,w'\in W:\left[\begin{matrix}wR^+_{R,S}w'=wRw'&\text{ and }&wR^-_{R,S}w'=wSw'\\[.3em]&\text{ and }&\\[.3em]wR^+_{S,R}w'=wSw'&\text{ and }&wR^-_{S,R}w'=wRw'\end{matrix}\right]
\end{align*}
A~formula $\phi\in\bimodalLtriangle$ is called \emph{bi-transferrable} iff the following condition holds for every pair of pointed frames $\langle\mathfrak{F},w\rangle$ and $\langle\mathfrak{F}',w\rangle$.
\begin{align*}
\mathfrak{F},w\models_{\KbiG}\phi\text{ and }\mathfrak{F}',w\models_{\KbiG}\phi&\text{ iff }\mathfrak{F}^{R,S},w\models_{\KGsquare}\phi\text{ and }\mathfrak{F}^{S,R},w\models_{\KGsquare}\phi
\end{align*}
\end{definition}
\begin{remark}\label{rem:transferbitransfer}
We draw the attention of our readers to the important distinction between transfer and bi-transfer. For transfer, we use \emph{the same frame}. In contrast to that, for bi-transfer, we need \emph{two mono-relational frames with the same carrier set} from which we construct \emph{two bi-relational frames with the same carrier set}: the one where $R^+$ is associated with $R$ and $R^-$ with $S$ and the one where the relations are inverted --- $R^+$ is associated with $S$ and $R^-$ with $R$.
\end{remark}

First, let us obtain the characterisation of (bi-)transferrable formulas.
\begin{theorem}\label{theorem:transfercriterion}
Let $\phi\in\bimodalLtriangle(2)$. Then, $\phi$ is (bi-)transferrable iff for every pointed frame $\langle\mathfrak{F},w\rangle$ it holds that
\begin{align}
\text{ if }\mathfrak{F},w\models_{\KbiG}\phi,\text{ then }\mathfrak{F},w\models_{\KbiG}{\sim}\phi^\partial\label{equ:transfercriterion}
\end{align}
\end{theorem}
\begin{proof}
We begin with the transferrable formulas. Let $\phi$ be transferrable and $\mathfrak{F},w\models_{\KbiG}\phi$ for some frame~$\mathfrak{F}$. Then, $\mathfrak{F},w\models_{\KGsquare}\phi$. Hence, $\mathfrak{F},w\models_{\KbiG}{\sim}\phi^\partial$, by Theorem~\ref{theorem:validityembedding}, as required. Now assume that \eqref{equ:transfercriterion} holds. If $\mathfrak{F},w\models_{\KbiG}\phi$, then $\mathfrak{F},w\models{\sim}\phi^\partial$. Thus, by Theorem~\ref{theorem:validityembedding}, we get that $\mathfrak{F},w\models_{\KGsquare}\phi$. If $\mathfrak{F},w\not\models_{\KbiG}\phi$, it is immediate that $\mathfrak{F}\not\models_{\KGsquare}\phi$. I.e., $\phi$ is transferrable.

The bi-transferrable formulas can be considered similarly. Let $\phi\in\bimodalLtriangle$ Assume that $\mathfrak{F}=\langle W,R\rangle$ and $\mathfrak{F}'=\langle W,S\rangle$. If $\phi$ is bi-transferrable, $\mathfrak{F},w\models\phi$, and $\mathfrak{F}',w\models_{\KbiG}\phi$ then $\mathfrak{F}^{R,S},w\models_{\KGsquare}\phi$ and $\mathfrak{F}^{S,R},w\models_{\KGsquare}\phi$. Hence, $\mathfrak{F}^{R,S},w\models_{\KGsquare}{\sim}\phi^\partial$ and $\mathfrak{F}^{S,R},w\models_{\KGsquare}{\sim}\phi^\partial$. Now recall from Convention~\ref{conv:monorelationallanguage} that $^\partial$ does not add new modalities when applied to $\bimodalLtriangle$ formulas. Since $\Box$ and $\lozenge$ in $\KbiG$ take into account only one relation (the first one), we have that $\mathfrak{F},w\models_{\KbiG}{\sim}\phi^\partial$ and $\mathfrak{F},w\models_{\KbiG}{\sim}\phi^\partial$, as required.

For the converse, assume that \eqref{equ:transfercriterion} holds, let $\mathfrak{F}$ and $\mathfrak{F}'$ be as in Definition~\ref{def:bitransfer}, $\mathfrak{F},w\models_{\KbiG}\phi$, and $\mathfrak{F}',w\models_{\KbiG}\phi$. Then $\mathfrak{F},w\models_{\KbiG}{\sim}\phi^\partial$ and $\mathfrak{F}',w\models_{\KbiG}{\sim}\phi^\partial$. It is clear that $\phi$ and ${\sim}\phi^\partial$ are $e_1$-valid (recall Definition~\ref{def:paraconsistentsemantics}) on $\langle\mathfrak{F}^{R,S},w\rangle$ and $\langle\mathfrak{F}^{S,R},w\rangle$. It remains to show that $\phi$ and ${\sim}\phi^\partial$ are $e_2$-valid as well.

It is easy to check by induction that the following statements hold.
\begin{itemize}
\item $\phi$ is $e_2$-valid on $\langle\mathfrak{F}^{R,S},w\rangle$ iff ${\sim}\phi^\partial$ is $e_1$-valid on $\langle\mathfrak{F}^{S,R},w\rangle$.
\item $\phi$ is $e_2$-valid on $\langle\mathfrak{F}^{S,R},w\rangle$ iff ${\sim}\phi^\partial$ is $e_1$-valid on $\langle\mathfrak{F}^{R,S},w\rangle$.
\item ${\sim}\phi^\partial$ is $e_2$-valid on $\langle\mathfrak{F}^{R,S},w\rangle$ iff $\phi$ is $e_1$-valid on $\langle\mathfrak{F}^{S,R},w\rangle$.
\item ${\sim}\phi^\partial$ is $e_2$-valid on $\langle\mathfrak{F}^{S,R},w\rangle$ iff $\phi$ is $e_1$-valid on $\langle\mathfrak{F}^{R,S},w\rangle$.
\end{itemize}
Thus, $\mathfrak{F}^{R,S},w\models_{\KGsquare}\phi$ and $\mathfrak{F}^{S,R},w\models_{\KGsquare}\phi$, as required. On the other hand, if $\mathfrak{F},w\not\models_{\KbiG}\phi$ (or $\mathfrak{F}',w\not\models_{\KbiG}\phi$, respectively), it is clear that $\mathfrak{F}^{R,S},w\not\models_{\KGsquare}\phi$ ($\mathfrak{F}^{S,R},w\not\models_{\KGsquare}\phi$).

The result now follows.
\end{proof}

Unfortunately, there seems to be no general way of establishing whether \eqref{equ:transfercriterion} holds for an arbitrary~$\phi$ since modal formulas encode \emph{second-order} conditions on the frame in the general case. Still, it is possible to show that some classes of formulas are transferrable. Namely, we will prove that Lemmon--Scott formulas are transferrable. To do this, we prove the (fuzzy analogue of) the Lemmon--Scott correspondence theorem for $\KG$.

Recall from~\cite[Theorems~4.1 and~4.2]{LemmonScott1977} that the Lemmon--Scott correspondence theorem for the classical modal logic states that a (crisp) frame $\mathfrak{F}=\langle W,R\rangle$ satisfies the first-order condition
\begin{align*}
\forall x,y,z\exists w:(xR^hy~\&~xR^jz)\Rightarrow(yR^iw~\&~zR^kw)
\end{align*}
iff $\lozenge^h\Box^ip\rightarrow\Box^j\lozenge^kp$ is (classically) valid on $\mathfrak{F}$. Here, $\Box^n$ and $\lozenge^n$ denote strings of $n$ $\Box$'s and $\lozenge$'s, respectively, while $R^n$ is a shorthand for the following composition of $R$ with itself: $\underbrace{R\circ\ldots\circ R}_{n\text{ times}}$ (with $wR^0w'$ standing for $w=w'$).
\begin{definition}[Compositions of fuzzy relations]\label{def:fuzzycomposition}
Let $R$ and $S$ be two fuzzy relations on $W$. We set
\begin{itemize}
\item $u(R\circ S)u'=\sup\{uRw\wedge_\mathsf{G}wSu':w\in W\}$;
\item $(R\circ S)(u)=\{u':u(R\circ S)u'>0\}$;
\item $uR^0u'=\begin{cases}0&\text{if }u\neq u'\\1&\text{if }u=u'\end{cases}$;
\item $uR^nu'=u(R\circ R^{n-1})u'$.
\end{itemize}
\end{definition}

It is clear that the following holds:
\begin{align*}
e(\Box^n\phi,w)&=\inf\{wR^nw'\rightarrow_\mathsf{G}e(\phi,w')\}&e(\lozenge^n\phi,w)&=\sup\{wR^nw'\wedge_\mathsf{G}e(\phi,w')\}
\end{align*}
\begin{theorem}[Fuzzy Lemmon--Scott correspondence]\label{theorem:FLS}
Let $\mathfrak{F}=\langle W,R\rangle$ be a fuzzy frame. Then
\begin{align*}
\mathfrak{F},x\models_{\KbiG}\lozenge^h\Box^ip\rightarrow\Box^j\lozenge^kp&\text{ iff }\underbrace{\forall y,z:(xR^hy\wedge_\mathsf{G}xR^jz)\leq\sup\limits_{w\in W}\{yR^iw\wedge_\mathsf{G}zR^kw\}}_{\mathsf{FLS}}
\end{align*}
\end{theorem}
\begin{proof}
Let $\mathfrak{F},x\not\models_{\KbiG}\lozenge^h\Box^ip\rightarrow\Box^j\lozenge^kp$. We prove that the condition on $R$ does not hold. We pick $e$~on $\mathfrak{F}$ s.t.\ $e(\lozenge^h\Box^ip,x)>e(\Box^j\lozenge^kp,x)$ and proceed as follows.
\begin{align*}
e(\lozenge^h\Box^ip,x)>e(\Box^j\lozenge^kp,x)&\text{ then }\sup\limits_{y'\in W}\{xR^hy'\wedge_{\mathsf{G}}e(\Box^ip,y')\}>\inf\limits_{z'\in W}\{xR^jz'\rightarrow_\mathsf{G}e(\lozenge^kp,z')\}\\
&\text{ then }\exists y',z':\left[\begin{matrix}xR^hy'>(xR^jz'\rightarrow_\mathsf{G}e(\lozenge^kp,z'))\\\text{and}\\e(\Box^ip,y')>(xR^jz'\rightarrow_\mathsf{G}e(\lozenge^kp,z'))\end{matrix}\right]\\
&\text{ then }\exists y',z':\left[\begin{matrix}xR^hy'>\sup\limits_{w\in W}\{z'R^kw\wedge_\mathsf{G}e(p,w)\}\\\text{and}\\xR^jz'>\sup\limits_{w\in W}\{z'R^kw\wedge_\mathsf{G}e(p,w)\}\\\text{ and }\\\inf\limits_{w\in W}\{y'R^iw\rightarrow_\mathsf{G}e(p,w)\}>\sup\limits_{w\in W}\{z'R^kw\wedge_\mathsf{G}e(p,w)\}\end{matrix}\right]\\
&\text{ then }\exists y',z':\left[\begin{matrix}xR^hy'>\left(\sup\limits_{w\in W}\{z'R^kw\}\wedge_\mathsf{G}\sup\limits_{w\in W}\{e(p,w)\}\right)\\\text{and}\\xR^jz'>\left(\sup\limits_{w\in W}\{z'R^kw\}\wedge_\mathsf{G}\sup\limits_{w\in W}\{e(p,w)\}\right)\\\text{ and }\\\inf\limits_{y'R^iw>e(p,w)}\{e(p,w)\}>\left(\sup\limits_{w\in W}\{z'R^kw\}\wedge_\mathsf{G}\sup\limits_{w\in W}\{e(p,w)\}\right)\end{matrix}\right]\\
&\text{ then }\underbrace{\exists y',z':\left[\begin{matrix}xR^hy'>\left(\sup\limits_{w\in W}\{z'R^kw\}\wedge_\mathsf{G}\sup\limits_{w\in W}\{e(p,w)\}\right)\\\text{and}\\xR^jz'>\left(\sup\limits_{w\in W}\{z'R^kw\}\wedge_\mathsf{G}\sup\limits_{w\in W}\{e(p,w)\}\right)\\\text{ and }\\\inf\limits_{y'R^iw>e(p,w)}\{e(p,w)\}>\sup\limits_{w\in W}\{z'R^kw\}\end{matrix}\right]}_{\Psi}
\end{align*}
It is clear that $\mathsf{FLS}$ and $\Psi$ are incompatible.

For the converse, let $\langle\mathfrak{F},x\rangle$ fail $\mathsf{FLS}$. Namely, let $x,y,z\in\mathfrak{F}$ be s.t.\ $(xR^hy\wedge_\mathsf{G}xR^jz)>\sup\limits_{w\in W}\{yR^iw\wedge_\mathsf{G}zR^kw\}$. It is clear that $(xR^hy\wedge_\mathsf{G}xR^jz)>\left(\sup\limits_{w\in W}\{yR^iw\}\wedge_\mathsf{G}\sup\limits_{w\in W}\{zR^kw\}\right)$. We set the valuation as follows: $e(p,w)=yR^iw$ for every $w\in R^i(y)$ and $e(p,u)=0$ otherwise. We have that $e(\Box^ip,y)=1$ and either
\begin{align}\label{equ:FLSeither}
e(\lozenge^kp,z)=\sup\limits_{w\in W}\{e(p,w)\}=\sup\limits_{w\in W}\{yR^iw\}<1&&\left(\text{if }\sup\limits_{w\in W}\{yR^iw\}<\sup\limits_{w\in W}\{zR^kw\}\right)
\end{align}
or
\begin{align}\label{equ:FLSor}
e(\lozenge^kp,z)=\sup\limits_{w\in W}\{zR^kw\}<1&&\left(\text{if }\sup\limits_{w\in W}\{yR^iw\}\geq\sup\limits_{w\in W}\{zR^kw\}\right)
\end{align}
If \eqref{equ:FLSeither} is the case, we have that $e(\Box^j\lozenge^kp,x)\leq\sup\limits_{w\in W}\{yR^iw\}<xR^hy$ and $e(\lozenge^h\Box^ip,x)\geq xR^hy$. If \eqref{equ:FLSor} holds, then $e(\Box^j\lozenge^kp,x)\leq\sup\limits_{w\in W}\{zR^kw\}$ but $e(\lozenge^h\Box^ip,x)\geq xR^hy>\sup\limits_{w\in W}\{zR^kw\}$. In both cases, we have that $e(\lozenge^h\Box^ip\rightarrow\Box^j\lozenge^kp,x)<1$, as required.
\end{proof}
Let us now use Theorem~\ref{theorem:FLS} to obtain the transfer of Lemmon--Scott formulas.
\begin{corollary}\label{cor:FLStransfer}~
\begin{enumerate}
\item Lemmon--Scott formulas are transferrable in mono-relational frames.
\item Lemmon--Scott formulas are bi-transferrable.
\end{enumerate}
\end{corollary}
\begin{proof}
By Theorem~\ref{theorem:transfercriterion}, it suffices to check that every Lemmon--Scott formula $\phi$ is $\KbiG$-valid on a~mono-relational pointed frame $\langle\mathfrak{F},w\rangle$ iff ${\sim}\phi^\partial$ is valid on $\langle\mathfrak{F},w\rangle$. Now, observe that $\mathfrak{F},w\models_{\KbiG}\phi$ iff $\mathfrak{F},w\models_{\KbiG}\triangle\phi$. Let $\triangle\phi=\triangle(\lozenge^h\Box^ip\rightarrow\Box^j\lozenge^kp)$.

It is clear that since $\mathfrak{F}$ is mono-relational, ${\sim}(\triangle(\lozenge^h\Box^ip\rightarrow\Box^j\lozenge^kp))^\partial={\sim\sim\sim}(\lozenge^j\Box^kp^*\coimplies\Box^h\lozenge^ip^*)$. This is equivalent (in $\biG$) to $\triangle(\lozenge^j\Box^kp^*\rightarrow\Box^h\lozenge^ip^*)$. Now, by Theorem~\ref{theorem:FLS}, we have that
\begin{align*}
\mathfrak{F},x\models_{\KbiG}\triangle(\lozenge^j\Box^kp^*\rightarrow\Box^h\lozenge^ip^*)&\text{ iff }\underbrace{\forall y,z:(xR^jy\wedge_\mathsf{G}xR^hz)\leq\sup\limits_{w\in W}\{yR^kw\wedge_\mathsf{G}zR^iw\}}_{\mathsf{FLS}'}
\end{align*}
It is clear that $\mathsf{FLS}$ and $\mathsf{FLS}'$ are equivalent conditions on frames. Thus, indeed, a Lemmon--Scott formula $\phi$ is $\KbiG$-valid on a pointed frame iff ${\sim}\phi^\partial$ is $\KbiG$-valid on that same pointed frame. The result now follows.
\end{proof}
\section{Complexity\label{sec:complexity}}
In this section, we establish the $\pspace$-completeness of the modal logics discussed in the paper. First, we tackle $\fuzzyKbiG$. Since we add only $\triangle$ (or $\coimplies$) which is a~propositional order-based connective, the proof is identical to the proof of the $\pspace$-completeness of $\fuzzyKG$ given in~\cite{CaicedoMetcalfeRodriguezRogger2017} which is why we will only state the required definitions and formulate the result.
\begin{definition}[$\mathsf{F}$-models of $\KbiG$]\label{def:Fmodels}
An $\mathsf{F}$-model is a tuple $\mathfrak{M}=\langle W,R,T,e\rangle$ with $\langle W,R,e\rangle$ being a~$\KbiG$ model and $T:W\rightarrow\mathcal{P}_{<\omega}([0,1])$ be s.t.\ $\{0,1\}\subseteq T(w)$ for all $w\in W$. $e$ is extended to the complex formulas as in $\KbiG$ in the cases of propositional connectives, and in the modal cases, as follows.
\begin{align*}
e(\Box\phi,w)&=\max\{x\in T(w):x\leq\inf\limits_{w'\in W}\{wRw'\rightarrow_\mathsf{G}e(\phi,w')\}\}\\
e(\lozenge\phi,w)&=\min\{x\in T(w):x\geq\sup\limits_{w'\in W}\{wRw'\wedge_\mathsf{G}e(\phi,w')\}\}
\end{align*}
\end{definition}

The next lemma is a straightforward extension of~\cite[Theorem~1]{CaicedoMetcalfeRodriguezRogger2013} to $\KbiG$. The proof is essentially the same since we add only $\triangle$ and $\coimplies$ (which are propositional connectives) to the language.
\begin{lemma}\label{lemma:FFMP}
$\phi$ is $\KbiG$-valid iff $\phi$ is true in all $\mathsf{F}$-models iff $\phi$ is true in all $\mathsf{F}$-models whose depth is $O(|\phi|)$ s.t.\ $|W|\leq(|\phi|+2)^{|\phi|}$ and $|T(w)|\leq|\phi|+2$ for all $w\in W$.
\end{lemma}

It is now clear that $\KbiG$ is decidable. To establish its complexity, we can utilise the algorithm described in~\cite{CaicedoMetcalfeRodriguezRogger2017}. The algorithm will work for $\KbiG$ since its only difference from $\KG$ is $\triangle$ and $\coimplies$ which are extensional connectives. Alternatively, we could expand the tableaux calculus for $\KG$ from~\cite{Rogger2016phd} with the rules for $\triangle$ and $\coimplies$ and use it to construct the decision procedure. Note furthermore, that in the presence of $\triangle$, validity and unsatisfiability are reducible to one another: $\phi$ is valid iff ${\sim}\triangle\phi$ is unsatisfiable; $\phi$ is unsatisfiable iff ${\sim}\phi$ is valid. The following statement is now immediate.
\begin{proposition}\label{prop:KbiGpspacecompleteness}
The validity of $\fuzzyKbiG$ is $\pspace$-complete.
\end{proposition}

Bi-relational $\mathsf{F}$-models can be introduced in the same way.
\begin{definition}[$\mathsf{F}$-models of $\KbiG(2)$]\label{def:F2models}
A bi-relational $\mathsf{F}$-model is a tuple $\mathfrak{M}=\langle W,R_1,R_2,T_1,T_2,e\rangle$ with $\langle W,R_1,R_2,e\rangle$ being a~$\KbiG(2)$ model and $T_1,T_2:W\rightarrow\mathcal{P}_{<\omega}([0,1])$ be s.t.\ $\{0,1\}\subseteq T_i(w)$ for all $w\in W$. $e$ is extended to the complex formulas as in $\KbiG$ in the cases of propositional connectives, and in the modal cases, as follows.
\begin{align*}
e(\Box_i\phi,w)&=\max\{x\in T_i(w):x\leq\inf\limits_{w'\in W}\{wR_iw'\rightarrow_\mathsf{G}e(\phi,w')\}\}\tag{$i\in\{1,2\}$}\\
e(\lozenge_i\phi,w)&=\min\{x\in T_i(w):x\geq\sup\limits_{w'\in W}\{wR_iw'\wedge_\mathsf{G}e(\phi,w')\}\}\tag{$i\in\{1,2\}$}
\end{align*}
\end{definition}

It is clear that Lemma~\ref{lemma:FFMP} holds w.r.t.\ bi-relational $\mathsf{F}$-models as well and that $\KbiG(2)$ is $\pspace$-complete too.
\begin{proposition}\label{prop:KbiG2pspacecompleteness}
The validity of $\KbiG(2)$ is $\pspace$-complete.
\end{proposition}

Finally, since the embeddings in Definition~\ref{def:translations} are linear, we obtain the $\pspace$-completeness of the paraconsistent logics.
\begin{proposition}\label{prop:birelPSpacecompleteness}
The validity of $\crispbirelKGsquare$, $\fuzzyKGsquare$, and $\fuzzybirelKGsquare$ is $\pspace$-complete.
\end{proposition}
\begin{proof}
The $\pspace$-membership follows immediately from Theorem~\ref{theorem:validityembedding} and Propositions~\ref{prop:KbiGpspacecompleteness} and~\ref{prop:KbiG2pspacecompleteness}. For the hardness, we proceed as follows. Let $\phi$ be a formula over $\{\mathbf{0},\wedge,\vee,\rightarrow,\Box\}$. We define $\phi^!$ to be the result of replacing every occurrence $p$ of variables with $\triangle p\wedge\neg{\sim}\triangle p$ and putting $\triangle$ in front of every~$\Box$. It is easy to establish by induction that $\phi^!$ can only be evaluated at $(1,0)$ or $(0,1)$.

We show that $\mathbf{K}\models\phi$ iff $\KGsquare\models\phi^!$. It is clear that $\KGsquare\not\models\phi^!$ when $\mathbf{K}\not\models\phi$ since classical values are preserved by $\bimodalLtrianglesquare$ connectives. For the converse, let $\mathfrak{M}=\langle W,R^+,R^-,e_1,e_2\rangle$ be a $\KGsquare$ model s.t.\ $e(\phi^!,w)=(0,1)$ for some $w\in W$. We construct a \emph{classical} model $\mathfrak{M}^!=\langle W,R^!,e^!\rangle$ as follows: $wR^!w'$ iff $wR^+w'>0$ or $wR^-w'>0$; $w\in e^!(p)$ iff $e(p,w)=(1,0)$. We check by induction that $e(\phi^!,w)=(1,0)$ iff $\mathfrak{M}^!,w\vDash\phi$. The basis case of $\phi^!=\triangle p\wedge\neg{\sim}\triangle p$ and $\phi=p$ holds by construction of $\mathfrak{M}^!$. The cases of propositional connectives can be proven directly from the induction hypothesis. Finally, if $\phi^!=\triangle\Box\psi^!$ and $\phi=\Box\psi$, we have that $e(\triangle\Box\psi^!,w)=(1,0)$ iff $e(\psi^!,w')=(1,0)$ for every $w'$ s.t.\ $wR^+w'>0$ or $wR^-w'>0$, which, by the induction hypothesis, is equivalent to $\mathfrak{M},w'\vDash\psi$ for every $w'\in R^!(w)$, and thus $\mathfrak{M},w\vDash\Box\psi$.
\end{proof}

To obtain the $\pspace$-completeness of $\infoGsquare$ is less straightforward since $\blacksquare$ and $\blacklozenge$ are not standard. To circumvent this, we use the approach from~\cite{BilkovaFrittellaKozhemiachenko2023nonstandard} and augment its language with an additional constant $\mathbf{B}$ s.t.\ $e_1(\mathbf{B},w)=e_2(\mathbf{B},w)=1$. We denote the resulting logics with $\infoGsquare(\mathbf{B})$.
\begin{proposition}\label{prop:infoPSpacecompleteness}~
\begin{enumerate}
\item The strong validity of all $\infoGsquare(\mathbf{B})$ logics is $\pspace$-complete.
\item $e_1$- and $e_2$-validities of all $\infoGsquare$ logics are $\pspace$-complete.
\end{enumerate}
\end{proposition}
\begin{proof}
Begin with 1. Membership follows from Theorem~\ref{theorem:validityembedding} and Propositions~\ref{prop:KbiGpspacecompleteness} and~\ref{prop:KbiG2pspacecompleteness}. We just need to add $\mathbf{B}^\partial=\mathbf{1}$ to the $^\partial$ translation in Definition~\ref{def:translations}. For the hardness, observe that since $e_1$-conditions from Definition~\ref{def:paraconsistentsemantics} coincide with the semantics of $\KbiG$ (Definition~\ref{def:KbiGsemantics}), we have immediately that $\phi$ is $\KbiG$-valid on a given (mono- or bi-relational) frame $\mathfrak{F}$ iff $\mathbf{B}\rightarrow\phi^{+\bullet}$ is strongly $\infoGsquare$-valid on $\mathfrak{F}$.

The proof of 2.\ is also simple. Again, the membership follows immediately from Theorem~\ref{theorem:validityembedding}, and Propositions~\ref{prop:KbiGpspacecompleteness} and~\ref{prop:KbiG2pspacecompleteness}. For the hardness, we observe that $\phi$ is $\KbiG$-valid iff $\phi^{+\bullet}$ is $e_1$-valid and that $\chi$ is $\KbiG$-valid iff ${\sim}\chi^\partial$ is $e_2$-valid.
\end{proof}

We finish the section with a short observation considering finitely-branching frames. Recall from~\cite{BilkovaFrittellaKozhemiachenko2022IJCAR,BilkovaFrittellaKozhemiachenko2023IGPL} that (both fuzzy and crisp) finitely-branching frames are definable in $\KbiG$. It is clear then that they are definable in $\infoGsquare(\mathbf{B})$ as well. In fact, one can construct a tableaux calculus for $\infoGsquare$ and $\infoGsquare(\mathbf{B})$ over finitely-branching frames (we refer the reader to~\cite{BilkovaFrittellaKozhemiachenko2023nonstandard}) and show that their satisfiabilities are also in $\pspace$ (in fact, the satisfiability of $\infoGsquare(\mathbf{B})$ is $\pspace$-complete).

Likewise, it is possible to define finitely-branching frames in $\birelKGsquare$.
\begin{proposition}\label{prop:finitebranching}
$\mathfrak{F}$ is finitely branching iff $\mathfrak{F}\models{\sim\sim}\Box(p\!\vee\!{\sim}p)$ and $\mathfrak{F}\models\mathbf{1}\!\coimplies\!\lozenge\neg(p\!\vee\!{\sim}p)$.
\end{proposition}
\begin{proof}
Observe that $e_1(p\vee{\sim}p,w)\!>\!0$ and $e_2(p\vee{\sim}p,w)\!<\!1$ for every $w\!\in\!\mathfrak{F}$. Since $\mathfrak{F}$ is finitely branching, $\inf\limits_{w'\in W}\{wSw'\rightarrow_{\mathsf{G}}e_1(p\vee{\sim}p,w')\}>0$ and $\sup\limits_{w'\in W}\{wSw'\wedge_{\mathsf{G}}e_2(p\vee{\sim}p,w')\}<1$ for $S\in\{R^+,R^-\}$. Thus, $e_1(\Box(p\vee{\sim}p),w)>0$ and $e_2(\Box(p\vee{\sim}p),w)<1$, whence, $e({\sim\sim}\Box(p\vee{\sim}p),w)=(1,0)$. Likewise, $e_1(\lozenge\neg(p\vee{\sim}p),w)<1$ and $e_2(\lozenge\neg(p\vee{\sim}p),w)>0$, whence $e(\mathbf{1}\coimplies\lozenge\neg(p\vee{\sim}p),w)=(1,0)$.

For the converse, we have two cases: (1) $|R^+(w)|\geq\aleph_0$ or (2) $|R^-(w)|\geq\aleph_0$ for some $w\in\mathfrak{F}$. In the first case, we let $X\subseteq R^+(w)$ be countable and define the value of $p$ as follows: $e(p,w'')=(1,0)$ for every $w''\notin X$ and $e(p,w_i)=\left(wR^+w'\cdot\frac{1}{i},0\right)$  for every $w_i\in X$. It is clear that $\inf\limits_{w'\in W}\{wR^+w'\rightarrow_\mathsf{G}e_1(p\vee{\sim}p,w')\}=0$, whence $e_1({\sim\sim}\Box(p\vee{\sim}p))=0$ as required.

In the second case, $Y\subseteq R^-(w)$ be countable and define the value of $p$ as follows: $e(p,w'')=(1,0)$ for every $w''\notin Y$ and $e(p,w_i)=\left(wR^-w'\cdot\frac{1}{i},0\right)$. It is clear that $\inf\limits_{w'\in W}\{wR^-w'\rightarrow_\mathsf{G}e_2(\neg(p\vee{\sim}p),w')\}=0$, whence $e_2(\lozenge\neg(p\vee{\sim}p))=0$ and $e_2(\mathbf{1}\coimplies\lozenge\neg(p\vee{\sim}p))=1$ as required.
\end{proof}
A tableaux calculus for $\KGsquare$ over finitely-branching frames can be constructed in a manner similar to the calculus for $\infoGsquare$ over finitely-branching frames presented in~\cite{BilkovaFrittellaKozhemiachenko2023nonstandard}.
\section{Conclusion\label{sec:conclusion}}
Let us summarise the results of the paper. We provided a strongly complete axiomatisation of the fuzzy bi-G\"{o}del modal logic (Theorem~\ref{theorem:KbiGstrongcompleteness}) and constructed the faithful embeddings of its paraconsistent relatives $\KGsquare$ and $\infoGsquare$ into $\KbiG$ and $\KbiG(2)$ depending on the number of relations in the frames. The embeddings hold for the valid formulas and for the valid entailments alike (Theorems~\ref{theorem:validityembedding} and~\ref{theorem:entailmentsembedding}). Using these embeddings, we provided a characterisation of $\KGsquare$- and $\infoGsquare$-definable classes of frames (Corollary~\ref{cor:definabilitycharacterisation}) as well as a characterisation of transferrable formulas (Theorem~\ref{theorem:transfercriterion}). Moreover, we established that all $\KGsquare$'s and $\infoGsquare$'s are $\pspace$-complete (Propositions~\ref{prop:birelPSpacecompleteness} and~\ref{prop:infoPSpacecompleteness}).

Still, several questions remain open. First of all, the axiomatisations of $\KGsquare$'s and $\infoGsquare$'s. Indeed, now we can only ‘prove’ negative normal forms of $\linebimodalLtrianglesquare$ and $\lineinfobimodalLtrianglesquare$ formulas in $\HKbiG(2)$ and then use the transformations in~\eqref{equ:lineNNFs} to obtain the intended formulas. However, since $\KGsquare$'s and $\infoGsquare$'s do not extend $\KbiG$ in the general case ($\infoGsquare$'s are non-standard, whence never extend $\KbiG$, while $\KGsquare$'s extend $\KbiG$ only for the crisp case), it would be instructive to obtain calculi designed specifically for these logics. One way to attempt this would be to construct a \emph{bi-lateral} calculus where the notion of proof is supplemented with the dual notion of disproof. Such calculi exist for bi-Intuitionistic logic (cf., e.g.,~\cite{Ayhan2021}) which is the bi-G\"{o}del logic without two prelinearity axioms $(p\rightarrow q)\vee(q\rightarrow p)$ and $\mathbf{1}\coimplies((p\coimplies q)\wedge(q\coimplies p))$. Likewise, there are bi-lateral calculi for the Belnap--Dunn logic (cf., e.g.,~\cite{OmoriWansing2022}) which can also be helpful since $\KGsquare$'s and $\infoGsquare$'s can be seen as Belnapian relatives of $\KbiG$.

Second, as we have already noted, the propositional fragment of $\KGsquare$'s and $\infoGsquare$'s is $\Gsquare$ --- a~certain paraconsistent expansion of $\mathsf{G}$. Namely, $\Gsquare$ is the linear extension of $\mathsf{I}_4\mathsf{C}_4$~\cite{Wansing2008} (cf.\ the axiomatisation of $\Gsquare$ in~\cite{BilkovaFrittellaKozhemiachenkoMajer2023IJAR}). It makes sense, then, to consider other modal expansions of paraconsistent logics and their linear extensions presented in~\cite{Wansing2008}. Moreover, it makes sense to investigate the modalities whose support of falsity \emph{dualises the support of truth}. In particular, for logics based on $\Gsquare$:
\begin{align*}
e_2(\Box\phi,w)&=\inf\limits_{w'\in W}\{e_2(\phi,w)\coimplies_\mathsf{G}wR^-w'\}&
e_2(\lozenge\phi,w)&=\inf\limits_{w'\in W}\{wR^-w'\vee_\mathsf{G}e_2(\phi,w)\}
\end{align*}

Third, we know from~\cite{BilkovaFrittellaKozhemiachenko2022IJCAR,BilkovaFrittellaKozhemiachenko2023IGPL} that \emph{crisp} $\KGsquare$ (both mono- and bi-relational) conservatively extend \emph{monorelational} $\crispKbiG$. Thus, there is a trivial embedding of $\KbiG$ into them. On the other hand, it is unclear whether $\fuzzyKbiG$ or $\fuzzyKbiG(2)$ can be embedded into \emph{fuzzy} $\KGsquare$. And although Corollary~\ref{cor:definabilitycharacterisation}, suggests that not all $\fuzzyKbiG$-definable classes of (mono- or bi-relational) frames are $\fuzzyKGsquare$-definable, it is not evident either. 

Finally, recall that $\KGsquare$ and $\infoGsquare$ semantics can be reformulated in terms of a single valuation on $[0,1]^{\Join}$ (cf.~Fig.~\ref{fig:01squareagain}). $\infoGsquare$ adds modalities that correspond to the infima and suprema w.r.t.\ the informational order on $[0,1]^{\Join}$. It makes sense, then, to combine $\KGsquare$ and $\infoGsquare$ into one logic and equip it with the full set of bi-lattice connectives. Note that the modal Belnapian bi-lattice logic from~\cite{JungRivieccio2013} turns out to be equivalent to the classical bi-modal logic~\cite{Speranski2022}. Our aim then is to determine whether this holds for $\KbiG(2)$ and the bi-lattice paraconsistent G\"{o}del modal logic. In particular, it will make sense to explore the connection between $\KbiG(2)$ enriched with the \L{}ukasiewicz negation $\sim_{\Luk}$ interpreted as $e({\sim_{\Luk}}\phi,w)=1-e(\phi,w)$\footnote{It is important to note that $\KbiG$ with the \L{}ukasiewicz negation \emph{does not coincide with $\KGsquare$}. Indeed, consider the following connectives
\begin{align*}
\triangletop\phi&\coloneqq\triangle\phi\wedge\triangle{\sim}\neg\phi
&
\triangletop_{\Luk}\phi&\coloneqq\triangle\phi\wedge\triangle{\sim\sim_{\Luk}}\phi
\end{align*}
It is clear that $\triangletop(p\rightarrow q)\vee\triangletop(q\rightarrow p)$ is not valid in $\Gsquare$ while $\triangletop_{\Luk}(p\rightarrow q)\vee\triangletop_{\Luk}(q\rightarrow p)$ is valid in $\KbiG$ with~${\sim_\Luk}$.} and the bi-lattice G\"{o}del modal logic with the conflation operator which is symmetric w.r.t.\ the \emph{vertical axis} of $[0,1]^{\Join}$.
\bibliographystyle{plain}
\bibliography{references.bib}
\end{document}